\documentclass{article}

\usepackage[T1]{fontenc}
\usepackage{times}
\usepackage{epsfig,graphicx,tabularx,color}
\usepackage[cp850]{inputenc}
\usepackage[english]{babel}
\usepackage{cite}
\usepackage{amsmath,amssymb,amsfonts,amsthm,mathrsfs,comment,mathdots,multirow,tikz}
\usepackage{rotating, longtable, relsize}
\usepackage[f]{esvect}
\usepackage{times}
\usepackage{fancyhdr}
\usepackage{lipsum}
\usepackage{nicefrac}

\usetikzlibrary{calc,3d}

\newtheorem{theorem}{Theorem}[section]

\newtheorem{lemma}[theorem]{Lemma}

\theoremstyle{definition}
\newtheorem{remark}[theorem]{Remark}

\def\XXint#1#2#3{{\setbox0=\hbox{$#1{#2#3}{\int}$}
     \vcenter{\hbox{$#2#3$}}\kern-.5\wd0}}

\def\ud{{\rm\,d}}

\def\C{\mathbb{C}}

\def\N{\mathbb{N}}

\def\R{\mathbb{R}}

\def\Z{\mathbb{Z}}

\def\OO{\mathcal{O}}

\def\pr(#1){\left({#1}\right)}
\def\br[#1]{\left[{#1}\right]}

\def\abs#1{\left|{#1}\right|}
\def\norm#1{\left\|{#1}\right\|}

\def\pFq#1#2{{\,}_{#1}F_{#2}}

\def\ii{{\rm i}}

\newcommand{\tr}{\operatorname{tr}}

\usetikzlibrary{arrows,positioning}

\begin{document}

\title{Fast and stable rational approximation of generalized hypergeometric functions}

\author{%
{\sc Richard Mika\"el Slevinsky\thanks{Email: Richard.Slevinsky@umanitoba.ca}}\\[2pt]
Department of Mathematics, University of Manitoba, Winnipeg, Canada}

\maketitle

\begin{abstract}
Rational approximations of generalized hypergeometric functions ${}_pF_q$ of type $(n+k,k)$ are constructed by the Drummond and factorial Levin-type sequence transformations. We derive recurrence relations for these rational approximations that require $\mathcal{O}[\max\{p,q\}(n+k)]$ flops. These recurrence relations come in two forms: for the successive numerators and denominators; and, for an auxiliary rational sequence and the rational approximations themselves. Numerical evidence suggests that these recurrence relations are much more stable than the original formul\ae~for the Drummond and factorial Levin-type sequence transformations. Theoretical results on the placement of the poles of both transformations confirm the superiority of factorial Levin-type transformation over the Drummond transformation.

 Factorial Levin-type sequence transformation.
\end{abstract}

\section{Introduction}

Generalized hypergeometric functions \cite[\S 16]{Olver-et-al-NIST-10} are formally equal to their Maclaurin series:
\begin{equation}\label{eq:pFqbyMac}
\pFq{p}{q}\left(\begin{array}{c}
\alpha_1,\ldots,\alpha_p\\
\beta_1,\ldots,\beta_q\end{array};
z\right) = \sum_{k=0}^\infty \dfrac{(\alpha_1)_k\cdots(\alpha_p)_k}{(\beta_1)_k\cdots(\beta_q)_k}\dfrac{z^k}{k!},
\end{equation}
where $(x)_n = \Gamma(x+n)/\Gamma(x)$ is the Pochhammer symbol for the rising factorial~\cite[\S 6.1.22]{Abramowitz-Stegun-65}. If some $\alpha_j \in -\N_0$, then the series terminates after at most $\alpha_j+1$ terms. The Maclaurin series has radius of convergence $0$ if $p > q+1$, $1$ if $p = q+1$, and $\infty$ if $p\le q$.

If $p\ge q+1$, the Maclaurin series may be analytically continued into an appropriately slit complex plane by the Meijer $G$-function~\cite[\S 16.18.1]{Olver-et-al-NIST-10} and its Mellin--Barnes integral representation~\cite[\S 16.17.1]{Olver-et-al-NIST-10}:
\begin{align}\label{eq:pFqbyG}
\pFq{p}{q}\left(\begin{array}{c}
\alpha_1,\ldots,\alpha_p\\
\beta_1,\ldots,\beta_q\end{array};
z\right) & := \frac{\displaystyle\prod_{k=1}^q\Gamma(\beta_k)}{\displaystyle\prod_{k=1}^p\Gamma(\alpha_k)}G_{p,q+1}^{1,p}\left(-z;\begin{array}{c}1-\alpha_1,\ldots,1-\alpha_p\\0,1-\beta_1,\ldots,1-\beta_q\end{array}\right),\\
& ~= \frac{1}{2\pi\ii}\int_C\frac{(\alpha_1)_s\cdots(\alpha_p)_s}{(\beta_1)_s\cdots(\beta_q)_s}\Gamma(-s)(-z)^s\ud s,
\end{align}
where the contour $C$ separates the poles of $(\alpha_k)_s$ for $k=1,\ldots,p$ from those of $\Gamma(-s)$ and is otherwise defined so that the integral converges. When $p = q+1$, the functions are analytic in the cut plane $\C\setminus[1,\infty)$, and when $p>q+1$, they are analytic in the cut plane $\C\setminus[0,\infty)$. In any case, we take Eq.~\eqref{eq:pFqbyG} to {\em define} the generalized hypergeometric functions and notice that the Maclaurin series is the asymptotic expansion of the contour integral as $z\to0$, whether the parameters and variable dictate that this expansions is convergent or divergent.

Generalized hypergeometric functions serve as an umbrella to most of the early transcendentals and special functions of mathematical physics; hence, they have many applications in a broad range of fields. Many numerical algorithms have been developed for their computation~\cite{Forrey-137-79-97,Becken-Schmelcher-126-449-00,Muller-90-179-01,Michel-Stoitsov-178-535-08,Colman-Cuyt-van-Deun-38-11-1-11,Willis-59-447-12,Doornik-84-1813-15,Pearson-Olver-Porter-74-821-17,Crespo-et-al-19,Johansson-45-30-1-19}, with a great majority focusing on Kummer's and Tricomi's confluent hypergeometric functions, Gauss' hypergeometric function, and others with $p+q\le3$, as a general observation. Moreover, several techniques in analytic continuation have helped extend our computational abilities of these special functions, including the re-expansion formul\ae~of B\"uhring~\cite{Buhring-18-884-87,Buhring-19-1249-88}, and Pad\'e approximants to some generalized hypergeometric functions by Iserles~\cite{Iserles-19-543-79} and Sidi~\cite{Sidi-7-37-81}.

In this work, we utilize the Drummond sequence transformation and a factorial Levin-type transformation to produce rational approximations of type $(n+k,k)$ to generalized hypergeometric functions. Sequence transformations are a powerful tool to accelerate the convergence of sequences (of partial sums) and, in case of divergence, to provide a reasonable antilimit. We refer the interested reader to an incomplete collection of references~\cite{Drummond-6-69-72,Levin-B3-371-73,Levin-Sidi-9-175-81,Weniger-10-189-89,Homeier-122-81-00,Sidi-03} on the matter and references therein for their rich history. The original formul\ae~for the rational approximations suffer from severe numerical instability with increasing $k$ that complicates any use of a reasonable stopping criterion and leads to unanticipated divergence in finite precision floating-point arithmetic. This numerical instability is due to two factors: that the rational function is represented as a polynomial divided by another polynomial instead of a ratio of rationals; and, that the polynomials are expressed in the monomial basis. Moreover, generating such sequences of rational approximations of type $(n+\ell,\ell)$ for $\ell = 0,\ldots,k$ incurs a computational complexity of $\OO[k(n+k)]$ flops.

The generalized hypergeometric series in Eq.~\eqref{eq:pFqbyMac} is special because the ratio of successive terms is a fixed rational function of the index. The converse is also true, whereby we relate the parameters to the negative roots and poles of that rational function of the index. This fact enables us to develop recurrence relations for the numerator and denominator sequences of Drummond and factorial Levin-type that cost $\OO[\max\{p,q\}(n+k)]$ flops to compute all rational approximations of type $(n+\ell,\ell)$ for $\ell = 0,\ldots,k$. By considering the ratios of successive denominators as an auxiliary rational sequence, recurrence relations are produced that represent the transformations as ratios of rationals. This stabilized scheme enables a reasonable stopping criterion to be effective at providing good results in finite precision.

We include a comprehensive suite of numerical examples to demonstrate the effectiveness of the new recurrence relations for the Drummond and factorial Levin-type transformations in a finite precision environment, and we briefly discuss its possible implementation in extended precision for guaranteed results. Finally, we include a discussion on the placement of the poles of the rational functions and extensions and applications of this work. Our method is implemented in the {\sc Julia} package {\tt HypergeometricFunctions.jl}~\cite{Slevinsky-GitHub-HypergeometricFunctionsjl}.

\section{Sequence transformations}

Consider the model sequence $\{a_n\}_{n=0}^\infty$, its partial sums:
\[
s_n = \sum_{k=0}^n a_k,
\]
and their limit $s = \lim_{n\to\infty} s_n$ (or antilimit in case of divergence~\cite{Sidi-03}) so that:
\[
s - s_n = r_n.
\]
A large constructive class of sequence transformations results from a characterization of the remainder $r_n$ in terms of a nonzero dominant controlling factor $\omega_n$ multiplied by a constructive refinement term $z_n^{(k)}$, resulting in the transformation:
\begin{equation}\label{eq:SequenceTransformation}
T_n^{(k)} - s_n = \omega_n z_n^{(k)}.
\end{equation}
The index $k$ refers to the number of unknowns appearing linearly in the constructive refinement. In particular, the equation may be sampled at $k+1$ distinct values of $n$, resulting in a $(k+1)\times(k+1)$ linear system of equations for the $k$ unknowns in $z_n^{(k)}$ and the value of the transformation $T_n^{(k)}$. As the linear system is dense, the computational complexity of this general approach is $\OO(k^3)$ flops.

To reduce the complexity, it has been noted that the values of $z_n^{(k)}$ are normally discarded; hence, for just the value of the transformation $T_n^{(k)}$, one expects to be able to reduce the complexity proportionally. In case the transformation is determined by $s_n, \ldots, s_{n+k}$, this can be done by dividing by the controlling factor $\omega_n$ and applying a so-called {\em annihilation} operator~\cite{Homeier-122-81-00} to the refinements $z_n^{(k)}$. We will now discuss two examples.

\subsection{Drummond's sequence transformation}

Drummond's sequence transformation~\cite{Drummond-6-69-72} is the result of using polynomial constructive refinements in Eq.~\eqref{eq:SequenceTransformation}, $z_n^{(k)} = p_{k-1}(n)$ where $p_{-1}(n) \equiv 0$ so that $\deg(p_{-1}) = 0$ and $\deg(p_{k-1}) = k-1$ otherwise. It is well-known that the forward finite difference operator $\Delta$ reduces the degree of a polynomial by one. Thus, $\Delta^k p_{k-1}(n) \equiv 0$.

Then:
\begin{equation}\label{eq:Drummond}
T_n^{(k)} = \frac{N_n^{(k)}}{D_n^{(k)}} = \frac{\displaystyle \Delta^k\left(\frac{s_n}{\omega_n}\right)}{\displaystyle \Delta^k\left(\frac{1}{\omega_n}\right)} = \frac{\displaystyle \sum_{j=0}^k \binom{k}{j}(-1)^{k-j}\frac{s_{n+j}}{\omega_{n+j}}}{\displaystyle \sum_{j=0}^k \binom{k}{j}(-1)^{k-j}\frac{1}{\omega_{n+j}}}
\end{equation}
For fixed $k$, Eq.~\eqref{eq:Drummond} may be computed in $\OO(n+k)$ flops, assuming $\omega_n$ costs $\OO(1)$ flops and $s_n$ as a partial sum costs $\OO(n)$ flops. It is important to generate multiple transformations for contiguous $k$ and fixed $n$ to have a better picture of the sequence of transformations of the partial sums. By using this formula for every $\ell = 0,\ldots,k$, this would cost $\OO[k(n+k)]$ flops, a quadratic complexity in $k$.

\subsection{A factorial Levin-type sequence transformation}

A factorial Levin-type transformation uses a constructive refinement term based on an inverse factorial series:
\begin{equation}\label{eq:inversefactorialLevin}
R_n^{(k)} - s_n = \omega_n \sum_{j=0}^{k-1} \dfrac{c_j}{(n+\gamma)_j},
\end{equation}
where $c_j$ are unknown constants and $\gamma>0$ is a usually small algorithmic parameter. This transformation first appeared in the construction of rational and Pad\'e approximants to certain generalized hypergeometric functions by Luke~\cite{Luke-38-279-60}, Fields~\cite{Fields-19-606-65}, and Sidi~\cite{Sidi-7-37-81}. It was cast as a sequence transformation by Shelef~\cite{Shelef-Thesis-87}, and it was mildly generalized to its current form by Weniger~\cite{Weniger-10-189-89} through the introduction of the algorithmic parameter $\gamma>0$. It is called a Levin-type transformation due to the seminal work by Levin~\cite{Levin-B3-371-73} where the constructive refinement terms are bounded as $n\to\infty$, as opposed to the polynomials in the Drummond transformation, enabling the controlling factor to take on its tautological role.

Multiplying both sides of the equation by $(n+\gamma)_{k-1}/\omega_n$, the right-hand side is again a polynomial that may be annihilated by the $k^{\rm th}$ order finite difference. This results in a transformation:
\begin{equation}\label{eq:factorialLevin}
R_n^{(k)} = \frac{P_n^{(k)}}{Q_n^{(k)}} = \frac{\displaystyle \Delta^k\left[\frac{(n+\gamma)_{k-1}s_n}{\omega_n}\right]}{\displaystyle \Delta^k\left[\frac{(n+\gamma)_{k-1}}{\omega_n}\right]} = \frac{\displaystyle \sum_{j=0}^k \binom{k}{j}(-1)^{k-j}\frac{(n+j+\gamma)_{k-1}s_{n+j}}{\omega_{n+j}}}{\displaystyle \sum_{j=0}^k \binom{k}{j}(-1)^{k-j}\frac{(n+j+\gamma)_{k-1}}{\omega_{n+j}}},
\end{equation}
that is a different weighted average of the partial sums. As with the Drummond transformation, Eq.~\eqref{eq:factorialLevin} may be computed in $\OO(n+k)$ flops and computing the first $\ell=0,\ldots,k$ transformations costs $\OO[k(n+k)]$ flops, also a quadratic complexity in $k$.

\subsection{Explicit remainder estimates}

The most reasonable choices of remainder estimates depend linearly on the sequence $\{a_n\}_{n=0}^\infty$ itself. This was formalized by Levin and Sidi~\cite{Levin-Sidi-9-175-81} for the class of sequences that satisfy an $m^{\rm th}$ order linear homogeneous difference equation with coefficients that have Poincar\'e-type asymptotic expansions as $n\to\infty$.

\begin{theorem}[Theorem 2 in~\cite{Levin-Sidi-9-175-81}]\label{theorem:mthorderdifference} Let:
\[
a_n = \sum_{k=1}^m p_k(n)\Delta^k a_n,\quad{\rm where}\quad p_k(n) \sim n^{i_k}\sum_{j=0}^\infty \frac{b_{k,j}}{n^j},\quad{\rm as}\quad n\to\infty,
\]
where $i_k\le k$, for $k=1,\ldots,m$ and suppose $\{s_n\}_{n=0}^\infty$ converges to $s$. If:
\[
\lim_{n\to\infty} \left[\Delta^{i-1} p_k(n)\right] \left[\Delta^{k-i} a_n\right] = 0,\quad{\rm for}\quad k=i,\ldots,m,\quad i = 1,\ldots,m,
\]
and if:
\[
\sum_{k=1}^m l(l-1)\cdots(l-k+1)p_{k,0} \ne 1,\quad{\rm for}\quad l\ge-1,
\]
where $p_{k,0} = \lim_{n\to\infty} n^{-k} p_k(n)$, then:
\[
s-s_n \sim \sum_{k=0}^{m-1} n^{j_k}\Delta^ka_n\sum_{j=0}^\infty\frac{c_{k,j}}{n^j},\quad{\rm as}\quad n\to\infty,
\]
where $j_k \le \max\{i_{k+1},i_{k+2}-1,\ldots,i_m-m+k+1\}$ for $k=0,\ldots,m-1$.
\end{theorem}

In case $m=1$, Theorem~\ref{theorem:mthorderdifference} describes a family of sequences whose remainders possess an asymptotic expansion that depends linearly on $a_n$. This theorem justifies the remainder estimates of inverse power series as in the original Levin transformation~\cite{Levin-B3-371-73}, as well as the remainder estimates of the inverse factorial series in Eq.~\eqref{eq:inversefactorialLevin}. Theorem~\ref{theorem:mthorderdifference} does not rigorously justify the remainder estimates in Drummond's transformation. Note that if $a_n$ satisfies the conditions of Theorem~\ref{theorem:mthorderdifference}, then we may take $\omega_n = n^{j_0} a_n = n^{j_0} \Delta s_{n-1}$, where $j_0\le 1$. Since the leading coefficients in the inverse power series may turn out to be zero, this is a subset of the choice $\omega_n = (n+\gamma) a_n$. Several remainder estimates have appeared in the literature~\cite{Levin-B3-371-73,Weniger-10-189-89,Homeier-Weniger-92-1-95}, including:
\begin{equation}\label{eq:Omega}
\omega_n = \left\{\begin{array}{c} a_n = \Delta s_{n-1},\\ a_{n+1} = \Delta s_n,\\ (n+\gamma) a_n = (n+\gamma) \Delta s_{n-1},\\ \dfrac{a_na_{n+1}}{a_{n+1}-a_n} = \dfrac{\Delta s_{n-1} \Delta s_n}{\Delta^2 s_{n-1}},\end{array}\right.
\end{equation}
among others.

In the context of generalized hypergeometric functions, we observe that:
\[
\frac{a_{n+1}}{a_n} = \frac{(\alpha_1+n)\cdots(\alpha_p+n)}{(\beta_1+n)\cdots(\beta_q+n)}\frac{z}{n+1},
\]
and it is easily verified that the sequence satisfies the conditions of Theorem~\ref{theorem:mthorderdifference} with $m=1$, namely a first-order linear homogeneous difference equation with a coefficient that has a Poincar\'e-type asymptotic expansion as $n\to\infty$. In this way, $R_n^{(k)}$ is a theoretically acceptable transformation for the class of convergent hypergeometric functions while $T_n^{(k)}$ is not. Nevertheless, we shall continue to consider both methods and offer another justification in \S~\ref{subsection:convergencetheorems}.

In this work, we describe algorithms to compute both the Drummond and factorial Levin-type transformations in linear complexity with respect to the transformation order $k$ when $\omega_n/\omega_{n+1}$ and $a_n/\omega_n$ are rational. As will be seen below, these two conditions imply that $a_{n+1}/a_n$ is rational. Thus, they are a specialization on the class of sequences with rational successive ratios whose controlling factors also have rational successive ratios.
\begin{lemma}
If $\omega_n/\omega_{n+1}$ and $a_n/\omega_n$ are rational, then $a_{n+1}/a_n$ is also rational.
\end{lemma}
\begin{proof}
This follows from:
\[
\dfrac{a_{n+1}}{a_n} = \dfrac{\dfrac{a_{n+1}}{\omega_{n+1}}}{\dfrac{a_n}{\omega_n}}\dfrac{\omega_{n+1}}{\omega_n}.
\]
\end{proof}
Recurrence relations have been developed~\cite{Weniger-10-189-89} to transform the sequences of partial sums $\{T_{n+\ell}^{(0)}\}_{\ell=0}^k$ and $\{R_{n+\ell}^{(0)}\}_{\ell=0}^k$ to the respective sequences $\{T_n^{(\ell)}\}_{\ell=0}^k$ and $\{R_n^{(\ell)}\}_{\ell=0}^k$. However, these only simplify the construction of binomial coefficients and do not reduce the complexity of the procedure.

\subsection{Convergence of the Drummond and Levin-type transformations applied to Euler's divergent series}\label{subsection:convergencetheorems}

We include here two remarkable results by Borghi and Weniger~\cite{Borghi-Weniger-94-149-15} on the resummation of Euler's divergent series:
\[
\sum_{k=0}^\infty k! z^k \sim \int_0^\infty \frac{e^{-t}}{1-zt}\ud t = \pFq{2}{0}(1,1;z),\quad{\rm for}\quad z\in\C\setminus(0,\infty).
\]
\begin{theorem}[Theorems 5.5 and 6.7 in~\cite{Borghi-Weniger-94-149-15}]\label{theorem:2F0convergence}
If $\omega_n = \Delta s_n$ and $\gamma=1$, then for every $z\in\C\setminus[0,\infty)$:
\begin{align*}
\pFq{2}{0}(1,1;z) - T_0^{(k)} & \sim \frac{2\pi}{z}\exp\left(\frac{-1}{z} - \frac{4k^{1/2}}{(-z)^{1/2}}\right),\\
\pFq{2}{0}(1,1;z) - R_0^{(k)} & \sim \frac{4\pi}{z}\exp\left(\frac{-1}{z} - \frac{9k^{2/3}}{2(-z)^{1/3}}\right)\cos\left(\frac{3^{3/2}k^{2/3}}{2(-z)^{1/3}}+\frac{\pi}{6}\right),\quad{\rm as}\quad k\to+\infty.
\end{align*}
\end{theorem}
On the negative real line, the dominant factors in the two convergence rates live in the argument of the exponential: Drummond's transformation has square root-exponential convergence while the factorial Levin-type transformation has an algebraic-exponential convergence closer to {\em bona fide} exponential convergence. Of secondary importance is that the constants in the convergence rates are $z$-dependent; in both cases, the convergence is faster for $z$ near the origin.

\subsection{Instability in the original formul\ae}

In the case of generalized hypergeometric functions, both transformations produce rational functions of type $(n+k,k)$ of the complex variable $z$ for any one of the remainder estimates in Eq.~\eqref{eq:Omega}. Direct use of Eqs.~\eqref{eq:Drummond} and~\eqref{eq:inversefactorialLevin} is cause for caution. For they both represent the rational approximations as ratios of polynomials in the monomial basis. However, there is no {\em a priori} reason to believe that the monomial basis is in any way optimal. If an appropriate remainder estimate is strictly alternating, then it is easy to confirm that the denominators in Eqs.~\eqref{eq:Drummond} and~\eqref{eq:inversefactorialLevin} are sums of terms of the same sign. This captures the best case scenario as subtractive cancellation may only occur in the numerators, and indeed it does if the sequence is strictly alternating. It turns out that this may still be quite severe because the numerators involve a double summation where inner summation is not simplified analytically when computing numerically. We call this a best case scenario because strict alternation in sign may only be true for evaluation of generalized hypergeometric functions on a semi-axis of the complex variable $z$, while we are naturally interested in the greater complex plane.

As will be seen in \S~\ref{section:linearcomplexity}, our fast algorithms define the numerators and denominators by recurrence relations whose lengths are independent of $n$ and $k$. Notwithstanding the reduction in quadratic to linear complexity, our fast algorithms have a profound impact on the numerical stability of the transformations as they use the successive numerators and denominators as their own polynomial bases. This is akin to computing orthogonal polynomials by their three-term recurrence relation as opposed to expansion in the monomial basis, which is unstable in general. While using the numerators and denominators as their own bases substantially improves of the stability of the algorithms, it does not guard against the real possibility of numerical overflow or underflow. Moreover, while rational functions in a compact subset of $\C$ may be bounded, their representative ratio of polynomials may have a large dynamic range. We further stabilize our algorithm in \S~\ref{subsubsection:PolynomialsToRationals} by converting our ratio of polynomials into a ratio of rationals, which has been observed in the AAA algorithm~\cite{Nakatsukasa-Sete-Trefethen-40-A1494-18} to be the best form for a rational approximant, among a ratio of polynomials, a ratio of rationals, and the partial fraction decomposition. With our stable ratio of rationals, a reasonable stopping criterion is proposed in \S~\ref{subsubsection:StoppingCriterion} that ascertains whether two successive transformations are numerically close. 

Figure~\ref{fig:pFqold2new} serves as a review of the aforementioned issues and a preview of their resolution by this work. In double precision, we consider the summation of a strictly alternating divergent series:
\begin{equation}\label{eq:2F011n2}
\pFq{2}{0}(1,1;-2) \approx 0.461\,455\,316\,241\,865\,234\ldots.
\end{equation}
where we have used a strictly alternating remainder estimate $\omega_n = a_{n+1}$, and parameters $\gamma = 2$ and $n=0$ in the Drummond and factorial Levin-type transformation. We observe that the original formul\ae~may come close to a good approximation, but soon the instabilities in Eqs.~\eqref{eq:Drummond} and~\eqref{eq:inversefactorialLevin} quickly lead to factorial divergence. On the other hand, the new methods demonstrate the algebraic-exponential convergence of Theorem~\ref{theorem:2F0convergence} with high fidelity and rounding errors that approximately accumulate at the much slower rate of $\OO(k\epsilon_{\rm mach})$.

\begin{figure}[htbp]
\begin{center}
\begin{tabular}{cc}
\includegraphics[width=0.485\textwidth]{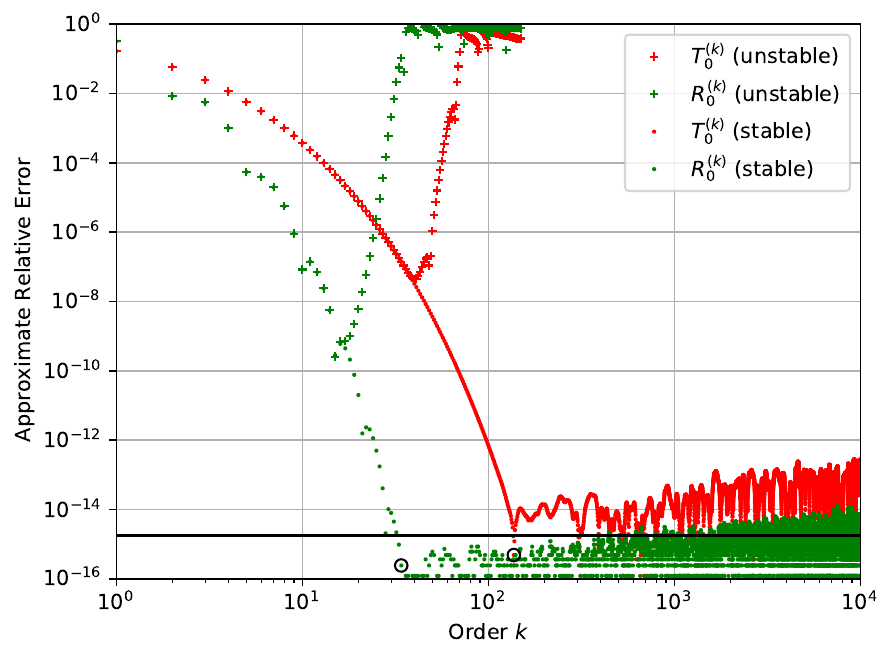}&
\includegraphics[width=0.485\textwidth]{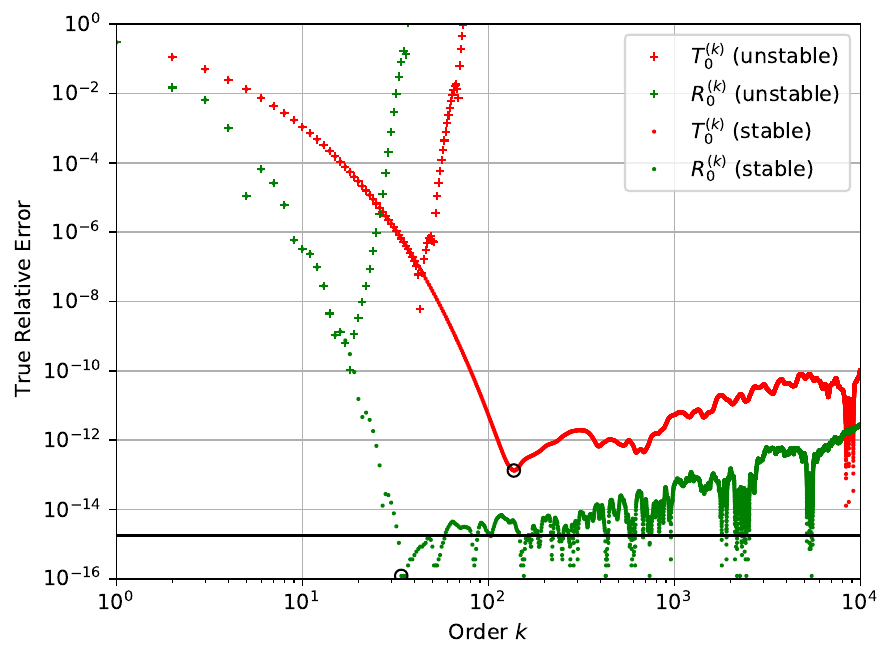}
\end{tabular}
\end{center}
\caption{Comparison of direct implementation of the Drummond and factorial Levin-type sequence transformations, labeled unstable, and the linear complexity recurrence relations derived in this work, labeled stable. In the left panel, the approximate relative error is that computed by using the next order transformation as the ``true'' value, whereas in the right panel the true relative error uses Eq.~\eqref{eq:2F011n2}, whose listed digits are correctly rounded (and truncated). In both panels, the black line indicates $8\epsilon_{\rm mach} \approx 1.77\times10^{-15}$ in double precision. The reasonable stopping criterion of \S~\ref{subsubsection:StoppingCriterion} selects the two values that are circled, $T_0^{(137)}$ and $R_0^{(34)}$, which are ostensibly, according to Eq.~\eqref{eq:StoppingCriterion}, the second values of the respective transformations to fall below the threshold of $8\epsilon_{\rm mach}$ in the left panel. As can be seen in the right panel, it provides no guarantee on the true relative error.}
\label{fig:pFqold2new}
\end{figure}

\section{Linear complexity recurrence relations}\label{section:linearcomplexity}

Recall the finite product rule:
\begin{equation}
\Delta^k(f_ng_n) = \sum_{j=0}^k \binom{k}{j}\Delta^j f_{n+k-j} \Delta^{k-j}g_n.
\end{equation}

\subsection{Drummond's sequence transformation}

\begin{theorem}\label{theorem:DrummondRecurrence}
Let $\omega_n/\omega_{n+1}$ and $a_n/\omega_n$ be rational. Then there exist recurrence relations for the numerator and denominators of Drummond's sequence transformation, whose length is independent of $n$ and $k$.
\end{theorem}
\begin{proof}
In the case of the denominator, we begin with the initial conditions:
\[
D_n^{(0)} = \dfrac{1}{\omega_n},\qquad D_{n+1}^{(0)} = \dfrac{1}{\omega_{n+1}},
\]
so that:
\[
D_{n+1}^{(0)} = \dfrac{\omega_n}{\omega_{n+1}}D_n^{(0)}.
\]
As $\omega_n/\omega_{n+1}$ is rational, we express it as the ratio of two polynomials, $q_n/p_n$, where $\deg(p_n) = p$ and $\deg(q_n) = q$, so that:
\[
p_nD_{n+1}^{(0)} = q_nD_n^{(0)}.
\]
Applying $\Delta^k$ to both sides and using $D_n^{(k+1)} = \Delta D_n^{(k)}$ we find:
\begin{equation}\label{eq:DrummondDenominatorRecurrence}
\sum_{j=0}^{\min\{k,p\}} \binom{k}{j} \Delta^j p_{n+k-j} D_{n+1}^{(k-j)} = \sum_{j=0}^{\min\{k,q\}} \binom{k}{j} \Delta^j q_{n+k-j} D_n^{(k-j)},
\end{equation}
where polynomial coefficients living in the kernel of sufficiently high order finite differences is cause for the minima in the summation indices.

Since, on the left-hand side, $D_{n+1}^{(k-j)} = \Delta D_n^{(k-j)} + D_n^{(k-j)} = D_n^{(k+1-j)} + D_n^{(k-j)}$, the recurrence relation only involves the subsequence $\{D_n^{(k-j)}\}_{j=-1}^{\min\{k,\max\{p,q\}\}}$, with fixed $n$.

As for the numerator sequence, the initial conditions read:
\[
N_n^{(0)} = \dfrac{s_n}{\omega_n},\qquad N_{n+1}^{(0)} = \dfrac{s_{n+1}}{\omega_{n+1}} = \dfrac{\omega_n}{\omega_{n+1}}\dfrac{s_n}{\omega_n} + \dfrac{a_{n+1}}{\omega_{n+1}},
\]
so that:
\[
N_{n+1}^{(0)} = \dfrac{\omega_n}{\omega_{n+1}}N_n^{(0)} + \dfrac{a_{n+1}}{\omega_{n+1}}.
\]
Thus, there exist polynomials $u_n$, $v_n$, and $w_n$ such that:
\[
u_nN_{n+1}^{(0)} = v_nN_n^{(0)} + w_n.
\]
By the same technique for the denominator sequence:
\begin{equation}\label{eq:DrummondNumeratorRecurrence}
\sum_{j=0}^{\min\{k,u\}} \binom{k}{j} \Delta^j u_{n+k-j} N_{n+1}^{(k-j)} = \sum_{j=0}^{\min\{k,v\}} \binom{k}{j} \Delta^j v_{n+k-j} N_n^{(k-j)} + \Delta^k w_n,
\end{equation}
where $\deg(u_n)=u$ and $\deg(v_n)=v$. The inhomogeneity is present for $k\le\deg(w_n) = w$.
\end{proof}


Some controlling factors allow one to relate $u_n$, $v_n$ and $w_n$ directly to $p_n$ and $q_n$. One such choice is the remainder estimate $\omega_n = \Delta s_n = a_{n+1}$, so that the numerator recurrence becomes:
\[
p_nN_{n+1}^{(0)} = q_nN_n^{(0)} + q_n.
\]
Another reasonable choice is the remainder estimate $\omega_n = \Delta s_{n-1} = a_n$. In this case, $a_{n+1}/\omega_{n+1} = 1$, and the numerator recurrence becomes:
\[
p_nN_{n+1}^{(0)} = q_nN_n^{(0)} + p_n.
\]
The other two choices in Eq.~\eqref{eq:Omega} are acceptable, though the polynomial coefficients are more involved.

\subsubsection{Computational considerations}

A direct use of the recurrence relations for the numerator and denominator sequences in Eqs.~\eqref{eq:DrummondDenominatorRecurrence} and~\eqref{eq:DrummondNumeratorRecurrence} would cost $\OO[\max\{p,q,u,v,w\}^2(n+k)]$ flops to compute the first $k+1$ terms, as a result of computing the binomial and variable coefficients, such as $\{\Delta^j p_{n+k-j}\}_{j=0}^p$, directly. We will now show that the product of binomial and variable coefficients may be obtained by an incremental calculation that reduces the overall complexity to $\OO[\max\{p,q,u,v,w\}(n+k)]$ flops. We will illustrate the method for the denominator recurrence as it applies equally to the numerator recurrence.

Defining:
\begin{equation}\label{eq:DrummondCoefficients}
p_{n,j}^{(k)} = \binom{k}{j}\Delta^j p_{n+k-j},\quad{\rm and}\quad q_{n,j}^{(k)} = \binom{k}{j}\Delta^j q_{n+k-j},
\end{equation}
then Eq.~\eqref{eq:DrummondDenominatorRecurrence} is equal to:
\[
\sum_{j=0}^{\min\{k,p\}} p_{n,j}^{(k)} D_{n+1}^{(k-j)} = \sum_{j=0}^{\min\{k,q\}} q_{n,j}^{(k)} D_n^{(k-j)}.
\]
By induction and the use of the finite difference operator, it follows that:
\begin{equation}\label{eq:DrummondCoefficientsByInduction}
p_{n,j}^{(k)} = p_{n+1,j}^{(k-1)} + \Delta p_{n, j-1}^{(k-1)},\quad{\rm and}\quad q_{n,j}^{(k)} = q_{n+1,j}^{(k-1)} + \Delta q_{n, j-1}^{(k-1)},
\end{equation}
where it is understood that $p_{n,j}^{(k)} = q_{n,j}^{(k)} = 0$ if $j < 0$ or $j>k$. However, more efficient recurrence relations for Eq.~\eqref{eq:DrummondCoefficients} may be utilized that reduce the flops and extra memory requirements:
\begin{equation}\label{eq:DrummondCoefficientsbyDeus}
j p_{n,j}^{(k)} = (k-j+1) p_{n,j-1}^{(k)} - k p_{n, j-1}^{(k-1)},\quad{\rm and}\quad j q_{n,j}^{(k)} = (k-j+1) q_{n,j-1}^{(k)} - k q_{n, j-1}^{(k-1)},
\end{equation}
Notice that Eq.~\eqref{eq:DrummondCoefficientsByInduction} requires evaluating $p_{n+1,j}^{(k-1)}$ and also $p_{n+1,j-1}^{(k-1)}$ through the finite difference, which requires the evaluation of $p_{n,j}^{(k)}$ for different values of $n$. On the other hand, once $p_{n,0}^{(k)} = p_{n+k}$ is computed, incrementing $j$ via Eq.~\eqref{eq:DrummondCoefficientsbyDeus} is a linear combination of two $p_{n,j}^{(k)}$ with the same index $n$. This ensures that the $k^{\rm th}$ row in the table of $p_{n,j}^{(k)}$ may be computed in $\OO(p)$ flops. We find that the same argument holds for $q_n$, $u_n$, and $v_n$.

The computational pattern in the numerical implementation of Eq.~\eqref{eq:DrummondCoefficientsbyDeus} is pictured in Figure~\ref{fig:DrummondComputationalPattern}.
\begin{figure}[htbp]
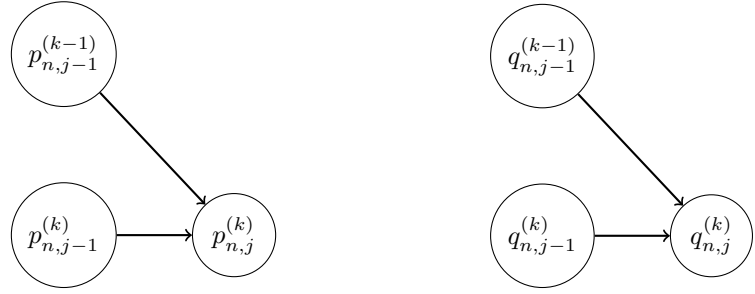

\begin{center}
\begin{tabular}{ccc}
\tikz {
  \node [circle,draw] (A) {$p_{n,j-1}^{(k-1)}$};
  \node [circle,draw] (B) [below=of A] {$p_{n,j-1}^{(k)}$};
  \node [circle,draw] (C) [right=of B] {$p_{n,j}^{(k)}$};
  \draw [draw = black, thick, ->]
    (A) edge (C)
    (B) edge (C);
}& \hspace{2cm} &
\tikz {
  \node [circle,draw] (A) {$q_{n,j-1}^{(k-1)}$};
  \node [circle,draw] (B) [below=of A] {$q_{n,j-1}^{(k)}$};
  \node [circle,draw] (C) [right=of B] {$q_{n,j}^{(k)}$};
  \draw [draw = black, thick, ->]
    (A) edge (C)
    (B) edge (C);
}\\
\end{tabular}
\caption{The computational pattern of Eq.~\eqref{eq:DrummondCoefficientsbyDeus}.}
\label{fig:DrummondComputationalPattern}
\end{center}
\end{figure}

\subsubsection{From a ratio of polynomials to a ratio of rationals}\label{subsubsection:PolynomialsToRationals}

As the resulting approximation is rational, we may rescale the sequence of numerators and denominators by the same power of the floating-point base to insure against overflow and underflow in the working precision to bring the magnitudes back into the representable range while not introducing any further rounding errors. This method, however, introduces a branch in the iteration and depends on floating-point implementation details, which are not universally portable. As such, we rely on another technique that also guards against overflow and underflow.

Consider $\mu_n^{(k)}$ to be the reciprocal ratio of successive denominators:
\[
\mu_n^{(k)} = \frac{D_n^{(k-1)}}{D_n^{(k)}},\qquad\mu_n^{(0)} = \frac{1}{D_n^{(0)}}.
\]
Then Eq.~\eqref{eq:DrummondDenominatorRecurrence} is equivalent to:
\[
\sum_{j=0}^{\min\{k,p\}} p_{n,j}^{(k)} \left(\frac{1}{\mu_n^{(k-j+1)}}+1\right)\mu_n^{(k)}\cdots\mu_n^{(k-j+1)} = \sum_{j=0}^{\min\{k,q\}} q_{n,j}^{(k)}\mu_n^{(k)}\cdots\mu_n^{(k-j+1)}.
\]
Notice that both of these summations may be computed by a generalized Horner scheme~\cite{Horner-109-308-19}:
\begin{align*}
&\sum_{j=0}^{\min\{k,p\}} p_{n,j}^{(k)} \left(\frac{1}{\mu_n^{(k-j+1)}}+1\right)\mu_n^{(k)}\cdots\mu_n^{(k-j+1)}\\ 
& = p_{n,0}^{(k)}\left(\frac{1}{\mu_n^{(k+1)}}+1\right) + \left(p_{n,1}^{(k)}\left(\mu_n^{(k)}+1\right) + \mu_n^{(k)}\left(p_{n,2}^{(k)}\left(\mu_n^{(k-1)}+1\right)+\cdots\right)\right),\\
& \sum_{j=0}^{\min\{k,q\}} q_{n,j}^{(k)}\mu_n^{(k)}\cdots\mu_n^{(k-j+1)} = q_{n,0}^{(k)} + \mu_n^{(k)}\left(q_{n,1}^{(k)} + \mu_n^{(k-1)}\left(q_{n,2}^{(k)} + \cdots\right)\right).
\end{align*}
With the ratio of successive denominators in hand, the numerator recurrence in Eq.~\eqref{eq:DrummondNumeratorRecurrence} may be transformed into a recurrence relation for the transformation $T_n^{(k)}$ itself:
\[
\sum_{j=0}^{\min\{k,u\}} u_{n,j}^{(k)} \left(\frac{T_n^{(k-j+1)}}{\mu_n^{(k-j+1)}}+T_n^{(k-j)}\right)\mu_n^{(k)}\cdots\mu_n^{(k-j+1)} = \sum_{j=0}^{\min\{k,v\}} v_{n,j}^{(k)}T_n^{(k-j)}\mu_n^{(k)}\cdots\mu_n^{(k-j+1)} + \left(\Delta^k w_n\right) \mu_n^{(k)}\cdots\mu_n^{(0)}.
\]
A generalized Horner scheme is also effective in the numerical evaluation of these sums.

To see why this formulation as a ratio of rationals is effective at deferring overflow, we observe that the more rapidly $\mu_n^{(k)}$ decays, the more damped are the successive terms in the summations. Therefore, a consistent asymptotic approximation has the leading term:
\[
\frac{1}{\mu_n^{(k+1)}} \sim \frac{q_{n,0}^{(k)}-p_{n,0}^{(k)}-p_{n,1}^{(k)}}{p_{n,0}^{(k)}},\quad{\rm as}\quad k\to\infty.
\]
As the coefficients are polynomials, the growth rate of this asymptotic approximation is substantially lower than the successive denominators themselves, which would grow factorially as they are recovered by the product of the successive ratios. In finite precision, this usually defers overflow far beyond an order $k$ where convergence sets in.

\subsubsection{A reasonable stopping criterion}\label{subsubsection:StoppingCriterion}

We have found that by virtue of the increased numerical stability in the linear complexity recurrence relations, a reasonable stopping criterion is useful in returning approximate numerical results. Given a small tolerance, say ${\tt tol} = 8\epsilon_{\rm mach}$, if:
\begin{equation}\label{eq:StoppingCriterion}
|T_n^{(k)}-T_n^{(k-1)}| \le \max\{|T_n^{(k)}|, |T_n^{(k-1)}|\} {\tt tol},
\end{equation}
then we return the value $T_n^{(k)}$ as the output of the algorithm. We also insure against false convergence of the algorithm, which we have found may occur for $k\le r+2$, where successive transformations may be numerically identical though are far from the limit of successive approximations.

There is of course no guarantee of convergence in the use of Eq.~\eqref{eq:StoppingCriterion}: in particular, the stopping criterion only implies that two successive transformations are relatively close. For any particularly ill-conditioned input sequence, rounding errors may be so amplified that the stopping criterion is never satisfied. Hence, it is necessary to also provide a maximum type of rational approximation, say $(n+k_{\rm max}, k_{\rm max})$ where $k_{\rm max} = 1,048,576$. Such a maximum type only affects the maximum computational time, not the memory requirements; hence, the setting of $k_{\rm max}$ is provided as an option with a default to the user. When the maximum type is reached, the output is returned with a warning. With such a stopping criterion, the algorithm may be easily adapted to return guaranteed results in an extended precision environment such as MPFR~\cite{Fousse-et-al-33-13-1-07} or Arb~\cite{Johansson-66-1281-17}: for a target result of $p$ correct bits, one runs the algorithm with $q>2p$ and $2q$ bits, doubling $q$, until the first $p$ bits of both results agree. As the method is generic to the data types of the parameters and the variable $z$, more sophisticated techniques involving interval and ball arithmetic can be used to increase the required internal precision by just as much as necessary.

\subsection{A factorial Levin-type sequence transformation}

\begin{theorem}\label{theorem:FactorialLevinRecurrence}
Let $\omega_n/\omega_{n+1}$ and $a_n/\omega_n$ be rational. Then there exist recurrence relations for the numerator and denominators of Levin's factorial sequence transformation, whose length is independent of $n$ and $k$.
\end{theorem}

Before the proof, we state a number of partial results. By Eq.~\eqref{eq:factorialLevin}, it follows that:
\begin{equation}\label{eq:WenigerQ0}
Q_n^{(0)} = \frac{1}{(n+\gamma-1)\omega_n},
\end{equation}
and:
\begin{equation}\label{eq:WenigerQFrom0}
Q_n^{(k)} = \Delta^k\left[(n+\gamma-1)_k Q_n^{(0)}\right].
\end{equation}

\begin{lemma}
Denominators in Levin's factorial sequence transformation satisfy:
\begin{equation}\label{eq:WenigerRecurrence}
Q_n^{(k+1)} = \left[(n+\gamma+2k)\Delta + (k+1)I\right] Q_n^{(k)},
\end{equation}
and:
\begin{equation}
\Delta^s Q_n^{(k+1)} = \left[(n+\gamma+2k+s)\Delta + (k+s+1)I\right]\Delta^sQ_n^{(k)},
\end{equation}
\end{lemma}
\begin{proof}
By Eq.~\eqref{eq:WenigerQFrom0}:
\begin{align*}
Q_n^{(k+1)} & = \Delta^{k+1}\left[(n+\gamma-1)_{k+1}Q_n^{(0)}\right],\\
& = \Delta^{k+1}\left[(n+\gamma+k-1)(n+\gamma-1)_kQ_n^{(0)}\right],\\
& = (n+\gamma+2k)\Delta^{k+1}\left[(n+\gamma-1)_kQ_n^{(0)}\right] + (k+1) \Delta^k\left[(n+\gamma-1)_kQ_n^{(0)}\right],\\
& = \left[(n+\gamma+2k)\Delta + (k+1)I\right]Q_n^{(k)}.
\end{align*}
The higher order recurrence is established by a similar observation. In particular:
\[
\Delta^s\left[(n+\gamma+2k)\Delta\right] = \left[(n+\gamma+2k+s)\Delta + sI\right]\Delta^s.
\]
\end{proof}

\begin{lemma}\label{lemma:InverseGeneralizedBinomialTransform}
Given a sequence $\{a_n\}_{n=0}^\infty$, define the sequence $\{b_n\}_{n=0}^\infty$ by the generalized binomial transform:
\begin{equation}
b_n = \sum_{k=0}^n\binom{n}{k}(c+n)_k a_k,\qquad c\notin-\N.
\end{equation}
Then the original sequence is recovered by:
\begin{equation}
a_n = \sum_{k=0}^n\binom{n}{k}(-1)^{n-k}\frac{c+2k}{(c+k)_{n+1}} b_k.
\end{equation}
\end{lemma}
\begin{proof}
Substituting $a_n$ from the second equation into the first:
\begin{align*}
b_n & = \sum_{k=0}^n \binom{n}{k}(c+n)_k\sum_{j=0}^k\binom{k}{j}(-1)^{k-j}\frac{c+2j}{(c+j)_{k+1}}b_j,\\
& = \sum_{j=0}^n b_j (c+2j)\sum_{k=j}^n\binom{n}{k}\binom{k}{j}(-1)^{k-j}\frac{(c+n)_k}{(c+j)_{k+1}}.
\end{align*}
Since $\displaystyle \binom{n}{k}\binom{k}{j} = \binom{n}{j}\binom{n-j}{k-j}$,
\begin{align*}
b_n & = \sum_{j=0}^n \binom{n}{j}b_j (c+2j)\sum_{k=j}^n\binom{n-j}{k-j}(-1)^{k-j}\frac{(c+n)_k}{(c+j)_{k+1}},\\
& = \sum_{j=0}^n \binom{n}{j}b_j (c+2j)\frac{(c+n)_j}{(c+j)_{j+1}}\sum_{k=0}^{n-j}\binom{n-j}{k}(-1)^k\frac{(c+n+j)_k}{(c+2j+1)_k},\\
& = \sum_{j=0}^n \binom{n}{j}b_j \frac{(c+n)_j}{(c+j)_j} \pFq{2}{1}\left(\begin{array}{c}j-n, c+n+j\\ c+2j+1\end{array}; 1\right).
\end{align*}
By the Chu--Vandermonde identity~\cite[\S 15.4.24]{Olver-et-al-NIST-10}:
\[
b_n = \sum_{j=0}^n \binom{n}{j}b_j \frac{(c+n)_j}{(c+j)_j}\frac{(j-n+1)_{n-j}}{(c+2j+1)_{n-j}} = \sum_{j=0}^n b_j \delta_{n,j},
\]
where $\delta_{n,j}$ is the Kronecker delta~\cite[\S 1]{Olver-et-al-NIST-10}.
\end{proof}

We are now in a position to relate higher order finite differences of $Q_n^{(k)}$ to successively indexed denominators (in $k$).
\begin{lemma}\label{lemma:LevinTransforms}
\begin{equation}\label{eq:LevinForwardTransform}
Q_n^{(k+r)} = \sum_{s=0}^r \binom{r}{s} (n+\gamma+2k+r-1)_s(k+s+1)_{r-s} \Delta^s Q_n^{(k)},
\end{equation}
\begin{equation}\label{eq:LevinBackwardTransform}
\Delta^r Q_n^{(k)} = \sum_{s=0}^r\binom{r}{s}(-1)^{r-s} \frac{n+\gamma+2k+2s-1}{(n+\gamma+2k+s-1)_{r+1}}(k+s+1)_{r-s} Q_n^{(k+s)}.
\end{equation}
\end{lemma}
\begin{proof}
We prove this lemma by induction. For $r=0$, Eq.~\eqref{eq:LevinForwardTransform} is an identity, and for $r=1$, Eq.~\eqref{eq:LevinForwardTransform} is equal to Eq.~\eqref{eq:WenigerRecurrence}. By Eq.~\eqref{eq:WenigerRecurrence} with $k\rightarrow k+r$:
\begin{align*}
Q_n^{(k+r+1)} & = \left[(n+\gamma+2k+2r)\Delta+(k+r+1)I\right]Q_n^{(k+r)}\\
& = (n+\gamma+2k+2r)\left[\sum_{s=0}^r\binom{r}{s} (n+\gamma+2k+r)_s (k+s+1)_{r-s} \Delta^{s+1}Q_n^{(k)}\right.\\
& \qquad\qquad\qquad   \left. + \sum_{s=0}^r\binom{r}{s} \Delta(n+\gamma+2k+r-1)_s (k+s+1)_{r-s} \Delta^sQ_n^{(k)}\right]\\
& \qquad + \sum_{s=0}^r\binom{r}{s} (n+\gamma+2k+r-1)_s (k+s+1)_{r+1-s} \Delta^sQ_n^{(k)}
\end{align*}
In the first summation, we substitute $s\to s+1$, and thanks to the binomial coefficients all three summations may be extended to share the same summation limits:
\begin{align*}
Q_n^{(k+r+1)} & = \sum_{s=0}^{r+1}\binom{r}{s-1} (n+\gamma+2k+2r)(n+\gamma+2k+r)_{s-1} (k+s)_{r+1-s} \Delta^sQ_n^{(k)}\\
& \qquad + \binom{r}{s} (n+\gamma+2k+2r) s(n+\gamma+2k+r)_{s-1} (k+s+1)_{r-s} \Delta^sQ_n^{(k)}\\
& \qquad + \binom{r}{s} (n+\gamma+2k+r-1)_s (k+s+1)_{r+1-s} \Delta^sQ_n^{(k)}.
\end{align*}
Since $(k+s)_{r+1-s} = (k+s)(k+s+1)_{r-s}$, the first two terms may be combined:
\begin{align*}
Q_n^{(k+r+1)} & = \sum_{s=0}^{r+1}\left[ \binom{r}{s-1}(k+s) + \binom{r}{s}s\right] (n+\gamma+2k+2r)(n+\gamma+2k+r)_{s-1} (k+s+1)_{r-s} \Delta^sQ_n^{(k)}\\
& \qquad + \binom{r}{s} (n+\gamma+2k+r-1)_s (k+s+1)_{r+1-s} \Delta^sQ_n^{(k)}.
\end{align*}
Using the binomial identity:
\[
\binom{r}{s-1}(k+s) + \binom{r}{s}s = \binom{r}{s-1}(k+r+1),
\]
we find:
\begin{align*}
Q_n^{(k+r+1)} & = \sum_{s=0}^{r+1}\binom{r}{s-1} (n+\gamma+2k+2r)(n+\gamma+2k+r)_{s-1} (k+s+1)_{r+1-s} \Delta^sQ_n^{(k)}\\
& \qquad + \binom{r}{s} (n+\gamma+2k+r-1)_s (k+s+1)_{r+1-s} \Delta^sQ_n^{(k)},\\
& = \sum_{s=0}^{r+1}\left[\binom{r}{s-1}(n+\gamma+2k+2r) + \binom{r}{s}(n+\gamma+2k+r-1)\right] (n+\gamma+2k+r)_{s-1} (k+s+1)_{r+1-s} \Delta^sQ_n^{(k)}.
\end{align*}
Two more binomial identities result in the equality:
\[
\binom{r}{s-1}(n+\gamma+2k+2r) + \binom{r}{s}(n+\gamma+2k+r-1) = \binom{r+1}{s}(n+\gamma+2k+r+s-1),
\]
so that:
\[
Q_n^{(k+r+1)} = \sum_{s=0}^{r+1}\binom{r+1}{s} (n+\gamma+2k+r)_s (k+s+1)_{r+1-s} \Delta^sQ_n^{(k)},
\]
which completes the proof of Eq.~\eqref{eq:LevinForwardTransform}. Since $(k+s+1)_{r-s} = \dfrac{\Gamma(k+r+1)}{\Gamma(k+s+1)}$ is a product of a function of $r$ times a function of $s$, the inverse formula in Eq.~\eqref{eq:LevinBackwardTransform} follows from Lemma~\ref{lemma:InverseGeneralizedBinomialTransform}.
\end{proof}

\begin{proof}[Proof of Theorem~\ref{theorem:FactorialLevinRecurrence}]
In the case of the denominator, we begin with the initial conditions:
\[
Q_n^{(0)} = \dfrac{1}{(n+\gamma-1)\omega_n},\qquad Q_{n+1}^{(0)} = \dfrac{1}{(n+\gamma)\omega_{n+1}},
\]
so that:
\[
(n+\gamma)Q_{n+1}^{(0)} = (n+\gamma-1)\dfrac{\omega_n}{\omega_{n+1}}Q_n^{(0)}.
\]
As $\omega_n/\omega_{n+1}$ is rational, we express it as the ratio of two polynomials, $q_n/p_n$, where $\deg(p_n) = p$ and $\deg(q_n) = q$, so that:
\[
(n+\gamma)p_nQ_{n+1}^{(0)} = (n+\gamma-1)q_nQ_n^{(0)}.
\]
For any natural number $r$, applying $\Delta^k\left[(n+\gamma)_{k-r-1}\diamond\right]$ and using $\hat{p}_n = (n+\gamma)p_n$, we find:
\begin{align}
& \Delta^k\left[(n+\gamma)_{k-r-1} \hat{p}_n Q_{n+1}^{(0)}\right] = \Delta^k\left[ (n+\gamma-1)_{k-r} q_n Q_n^{(0)} \right],\\
& \sum_{j=0}^{\min\{k,p+1\}}\binom{k}{j} \Delta^j\hat{p}_{n+k-j} \Delta^{k-j}\left[(n+\gamma)_{k-r-1}Q_{n+1}^{(0)}\right]\nonumber\\
& \qquad = \sum_{j=0}^{\min\{k,q\}}\binom{k}{j}\Delta^jq_{n+k-j} \Delta^{k-j}\left[(n+\gamma-1)_{k-r}Q_n^{(0)}\right],\\
& \sum_{j=0}^{\min\{k,p+1\}}\binom{k}{j} \Delta^j\hat{p}_{n+k-j} \Delta^{r+1-j} Q_{n+1}^{(k-r-1)} = \sum_{j=0}^{\min\{k,q\}}\binom{k}{j}\Delta^jq_{n+k-j} \Delta^{r-j} Q_n^{(k-r)}.\label{eq:WenigerDenominatorFiniteRecurrence}
\end{align}
By Lemma~\ref{lemma:LevinTransforms}, we observe that higher order finite differences of denominators are linear combinations of successive denominators. Thus, Eq.~\eqref{eq:WenigerDenominatorFiniteRecurrence} only involves the subsequence $\{Q_n^{(k-j)}\}_{j=-1}^{r+1}$. We shall return to this observation in \S~\ref{subsubsection:LevinComputational}.

As for the numerator sequence, the initial conditions read:
\[
P_n^{(0)} = \dfrac{s_n}{(n+\gamma-1)\omega_n},\qquad P_{n+1}^{(0)} = \dfrac{s_{n+1}}{(n+\gamma)\omega_{n+1}},
\]
so that:
\[
(n+\gamma)P_{n+1}^{(0)} = (n+\gamma-1)\dfrac{\omega_n}{\omega_{n+1}}P_n^{(0)} + \dfrac{a_{n+1}}{\omega_{n+1}}.
\]
As $\omega_n/\omega_{n+1}$ and $a_{n+1}/\omega_{n+1}$ are rational, there exist polynomials $u_n$, $v_n$, and $w_n$ such that:
\[
(n+\gamma)u_nP_{n+1}^{(0)} = (n+\gamma-1)v_nP_n^{(0)} + w_n.
\]
Using the same technique as for the denominator sequence, we define $\hat{u}_n = (n+\gamma)u_n$ and:
\begin{equation}\label{eq:WenigerNumeratorFiniteRecurrence}
\sum_{j=0}^{\min\{k,u+1\}}\binom{k}{j} \Delta^j\hat{u}_{n+k-j} \Delta^{r+1-j} P_{n+1}^{(k-r-1)} = \sum_{j=0}^{\min\{k,v\}}\binom{k}{j}\Delta^jv_{n+k-j} \Delta^{r-j} P_n^{(k-r)} + \Delta^k\left[(n+\gamma)_{k-r-1}w_n\right].
\end{equation}
If we take $r = r^* := \max\{p+1, q, u+1, v, w\}$, then for every $k\ge r^*+1$:
\[
\deg((n+\gamma)_{k-r^*-1} w_n) = k-r^*-1+w \le k-1,
\]
which shows that the presence of the inhomogeneity in the recurrence relation for the numerator sequence is evanescent and both recurrence relations in Eqs.~\eqref{eq:WenigerDenominatorFiniteRecurrence} and~\eqref{eq:WenigerNumeratorFiniteRecurrence} have length independent of $n$ and $k$. 

If $\gamma$ is a natural number, for sufficiently small $n+k$ and sufficiently large $r$, the inhomogeneity, $\Delta^k[(n+\gamma)_{k-r-1}w_n]$ will be singular through the Pochhammer symbol's poles. Moreover, if $\gamma$ is in the neighbourhood of a natural number, then the issue becomes one of numerical instability as opposed to outright singularity. To avoid this problem, for $k\le r^*$, the recurrence relations employed with $r = k-1$ have no singularity present in the inhomogeneity.
\end{proof}
\begin{remark}\label{remark:reducelength}
We have found that a special case arises in practice that reduces the length of the recurrence relations by one. If $n+\gamma$ divides $q_n$, $v_n$ and $w_n$, then the recurrence relations are simplified by letting $q_n = (n+\gamma)\hat{q}_n$, $v_n = (n+\gamma)\hat{v}_n$, $w_n = (n+\gamma)\hat{w}_n$, and $r = r^\star := \max\{p,q-1,u,v-1,w-1\}$ to:
\begin{align}
& \sum_{j=0}^{\min\{k,p\}}\binom{k}{j} \Delta^jp_{n+k-j} \Delta^{r+1-j} Q_{n+1}^{(k-r-1)} = \sum_{j=0}^{\min\{k,q-1\}}\binom{k}{j}\Delta^j\hat{q}_{n+k-j} \Delta^{r-j} Q_n^{(k-r)},\label{eq:WenigerDenominatorSimplifiedFiniteRecurrence}\\
& \sum_{j=0}^{\min\{k,u\}}\binom{k}{j} \Delta^ju_{n+k-j} \Delta^{r+1-j} P_{n+1}^{(k-r-1)} = \sum_{j=0}^{\min\{k,v-1\}}\binom{k}{j}\Delta^j\hat{v}_{n+k-j} \Delta^{r-j} P_n^{(k-r)} + \Delta^k\left[(n+\gamma)_{k-r-1}\hat{w}_n\right].\label{eq:WenigerNumeratorSimplifiedFiniteRecurrence}
\end{align}
\end{remark}

\subsubsection{Computational considerations}\label{subsubsection:LevinComputational}

By using Eq.~\eqref{eq:LevinBackwardTransform} in Eqs.~\eqref{eq:WenigerDenominatorFiniteRecurrence} and~\eqref{eq:WenigerNumeratorFiniteRecurrence}, exchanging the order of summations yields:
\begin{align*}
\sum_{j=0}^{r+1} \hat{p}_{n,j}^{(k)}(r) (n+\gamma+2k-2j)Q_{n+1}^{(k-j)} & = \sum_{j=0}^r q_{n,j}^{(k)}(r) (n+\gamma+2k-2j-1)Q_n^{(k-j)},\\
\sum_{j=0}^{r+1} \hat{u}_{n,j}^{(k)}(r) (n+\gamma+2k-2j)P_{n+1}^{(k-j)} & = \sum_{j=0}^r v_{n,j}^{(k)}(r) (n+\gamma+2k-2j-1)P_n^{(k-j)} + \Delta^k\left[(n+\gamma)_{k-r-1}w_n\right].
\end{align*}
where:
\begin{align*}
\hat{p}_{n,j}^{(k)}(r) & = \sum_{s=0}^{\min\{k,j,p+1\}}\binom{k}{s}\binom{r+1-s}{r+1-j}\frac{(-1)^{j-s}(k-j+1)_{j-s}}{(n+\gamma+2k-j-r-1)_{r-s+2}}\Delta^s\hat{p}_{n+k-s},\\
q_{n,j}^{(k)}(r) & = \sum_{s=0}^{\min\{k,j,q\}}\binom{k}{s}\binom{r-s}{r-j}\frac{(-1)^{j-s}(k-j+1)_{j-s}}{(n+\gamma+2k-j-r-1)_{r-s+1}}\Delta^sq_{n+k-s},\\
\hat{u}_{n,j}^{(k)}(r) & = \sum_{s=0}^{\min\{k,j,u+1\}}\binom{k}{s}\binom{r+1-s}{r+1-j}\frac{(-1)^{j-s}(k-j+1)_{j-s}}{(n+\gamma+2k-j-r-1)_{r-s+2}}\Delta^s\hat{u}_{n+k-s},\\
v_{n,j}^{(k)}(r) & = \sum_{s=0}^{\min\{k,j,v\}}\binom{k}{s}\binom{r-s}{r-j}\frac{(-1)^{j-s}(k-j+1)_{j-s}}{(n+\gamma+2k-j-r-1)_{r-s+1}}\Delta^sv_{n+k-s}.
\end{align*}
The beauty of these formul\ae~is that we may derive recurrence relations between adjacent entries. For $\hat{p}$, we find:
\begin{align}
& j\hat{p}_{n,j}^{(k)}(r)+(n+\gamma+2k-j-r-2)k\hat{p}_{n,j-1}^{(k-1)}(r)\label{eq:WenigerPbyDeus}\\
& = (k-j+1)\left[(n+\gamma+2k-j+1)\hat{p}_{n,j-1}^{(k)}(r) - (r-j+3)k\hat{p}_{n,j-2}^{(k-1)}(r)\right],\nonumber
\end{align}
in general and for $r=k-1$:
\begin{align}
& j\hat{p}_{n,j}^{(k)}(k-1)+k\hat{p}_{n,j-1}^{(k-1)}(k-2)\label{eq:WenigerPbyDeus2}\\
& = (k-j+1)\left[(n+\gamma+2k-j+1)\hat{p}_{n,j-1}^{(k)}(k-1) + k\hat{p}_{n,j-2}^{(k-1)}(k-2)\right].\nonumber
\end{align}
For $q$, we find:
\begin{align}
& jq_{n,j}^{(k)}(r)+(n+\gamma+2k-j-r-2)kq_{n,j-1}^{(k-1)}(r)\label{eq:WenigerQbyDeus}\\
& = (k-j+1)\left[(n+\gamma+2k-j)q_{n,j-1}^{(k)}(r) - (r-j+2)kq_{n,j-2}^{(k-1)}(r)\right],\nonumber
\end{align}
in general and for $r=k-1$:
\begin{align}
& jq_{n,j}^{(k)}(k-1)+kq_{n,j-1}^{(k-1)}(k-2)\label{eq:WenigerQbyDeus2}\\
& = (k-j+1)\left[(n+\gamma+2k-j)q_{n,j-1}^{(k)}(k-1) + kq_{n,j-2}^{(k-1)}(k-2)\right].\nonumber
\end{align}
By analogy, recurrence relations with the same structures hold for $\hat{u}_{n,j}^{(k)}$ and $v_{n,j}^{(k)}$, respectively. Once $k = r^*+1$, we notice that $\hat{p}_{n,j}^{(k)}(r^*) \equiv \hat{p}_{n,j}^{(k)}(k-1)$ and $q_{n,j}^{(k)}(r^*) \equiv q_{n,j}^{(k)}(k-1)$ and we may switch from Eqs.~\eqref{eq:WenigerPbyDeus2} and~\eqref{eq:WenigerQbyDeus2} to Eqs.~\eqref{eq:WenigerPbyDeus} and~\eqref{eq:WenigerQbyDeus}, respectively, where there are no singularities in the coefficients and the order $k$ is sufficiently large that the numerator recurrence is homogeneous.

The computational pattern in the numerical implementation of Eqs.~\eqref{eq:WenigerPbyDeus}--\eqref{eq:WenigerQbyDeus2} is pictured in Figure~\ref{fig:WenigerComputationalPattern}.
\begin{figure}[htbp]
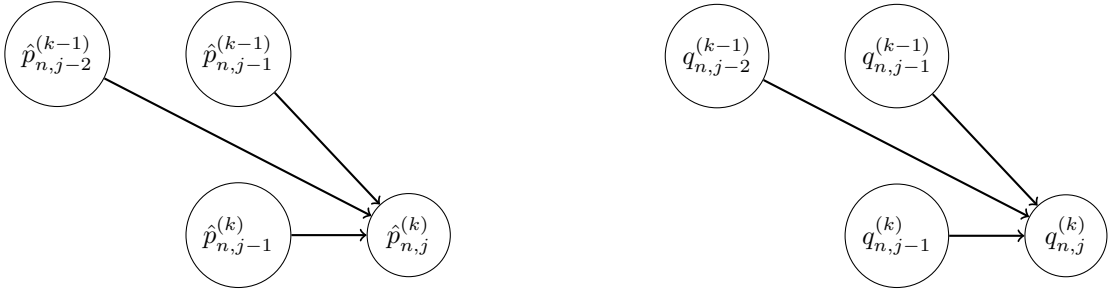

\begin{center}
\begin{tabular}{ccc}
\tikz {
  \node [circle,draw] (A)              {$\hat{p}_{n,j-2}^{(k-1)}$};
  \node [circle,draw] (B) [right=of A] {$\hat{p}_{n,j-1}^{(k-1)}$};
  \node [circle,draw] (C) [below=of B] {$\hat{p}_{n,j-1}^{(k)}$};
  \node [circle,draw] (D) [right=of C] {$\hat{p}_{n,j}^{(k)}$};
  \draw [draw = black, thick, ->]
    (A) edge (D)
    (B) edge (D)
    (C) edge (D);
}
& \hspace{2cm} &
\tikz {
  \node [circle,draw] (A)              {$q_{n,j-2}^{(k-1)}$};
  \node [circle,draw] (B) [right=of A] {$q_{n,j-1}^{(k-1)}$};
  \node [circle,draw] (C) [below=of B] {$q_{n,j-1}^{(k)}$};
  \node [circle,draw] (D) [right=of C] {$q_{n,j}^{(k)}$};
  \draw [draw = black, thick, ->]
    (A) edge (D)
    (B) edge (D)
    (C) edge (D);
}
\\
\end{tabular}
\caption{The computational pattern of Eqs.~\eqref{eq:WenigerPbyDeus}--\eqref{eq:WenigerQbyDeus2}.}
\label{fig:WenigerComputationalPattern}
\end{center}
\end{figure}

The method benefits from the following scaling of the numerators and denominators:
\[
P_n^{(k)} = (n+\gamma)_{k-1}\tilde{P}_n^{(k)},\quad{\rm and}\quad Q_n^{(k)} = (n+\gamma)_{k-1}\tilde{Q}_n^{(k)}.
\]
Furthermore, as with the Drummond transformation, the recurrence relations for the numerator and denominator sequences of polynomials can be converted into recurrence relations for an auxiliary rational sequence and the sequence of transformations. We elide the description as it would be identical to \S~\ref{subsubsection:PolynomialsToRationals}. Combined, these two observations defer numerical overflow to an order $k$ that far exceeds what is used in practice. We use the same reasonable stopping criterion described in \S~\ref{subsubsection:StoppingCriterion} in our implementation of this transformation.

\section{Numerical results}

In this section, we use the linear complexity recurrence relations for the Drummond and factorial Levin-type sequence transformations to approximate generalized hypergeometric functions $\pFq{p}{q}(\alpha,\beta;z)$. All experiments are performed on an iMac Pro (Early 2018) with a 2.3 GHz Intel Xeon W-2191B processor. Our free and open-source implementation lives in the {\sc Julia} package {\tt HypergeometricFunctions.jl}~\cite{Slevinsky-GitHub-HypergeometricFunctionsjl}.

By choosing any one of the remainder estimates in Eq.~\eqref{eq:Omega} or others in the same spirit, different families of rational approximations may be generated. We have found the differences to be quite subtle, and therefore we will present all of our numerical results with a single choice of remainder estimate --- the original choice in the Drummond transformation~\cite{Drummond-6-69-72}, and one that has also been proposed elsewhere~\cite{Smith-Ford-16-223-79,Weniger-10-189-89,Borghi-Weniger-94-149-15}:
\begin{equation}
\omega_n = \Delta s_n = a_{n+1}.
\end{equation}
Moreover, we only present numerical results for $n=0$, computing the type $(k,k)$ rational approximations. As with Pad\'e approximants, more power is usually gained with equal degree numerator and denominator than by fixing the degree of the denominator. In the factorial Levin-type transformation, we have found that the parameter $\gamma$ has only a small impact on the results, provided it is not too large. In the numerical experiments, we set $\gamma=2$.

As $p$ and $q$ in $\pFq{p}{q}$ increase, we are faced with an increasing dimension of the parameter space. While randomization in testing may increase confidence in the output, we believe that the parameter space is still too large for randomization to yield anything more meaningful than deterministic tests. We also wish to view the errors in the complex $z$-plane; hence, deterministic and reproducible values of the parameters are elected. Given the above restrictions, any numerical experiments with sequence transformations approximating generalized hypergeometric functions can only be indicative rather than exhaustive. 

Apart from terminating series (polynomials), generalized hypergeometric functions $\pFq{p}{q}$ are: entire when $p\le q$; analytic in the cut plane $\C\setminus[1,\infty)$ when $p=q+1$; or, analytic in the cut plane $\C\setminus[0,\infty)$ when $p>q+1$. As $\pFq{0}{0}(;z) \equiv e^z$ and $\pFq{1}{0}(\alpha;z) \equiv (1-z)^{-\alpha}$, we shall use $\pFq{1}{1}$, $\pFq{2}{1}$, and $\pFq{2}{0}$ as nontrivial numerical representatives of each of these respective classes. For the purposes of calculating errors, we shall take the approximations $R_0^{(k)}$ obtained in extended $128$-bit precision as true values, while the approximations under examination are calculated in $64$-bit precision; in the complex plane, we are using a tuple of equally precise floating-point numbers for the real and imaginary parts. Each of Figures~\ref{fig:1F1Drummond}--\ref{fig:2F0Weniger} contains four panels illustrating, clockwise from top left: the phase portrait of the transformations at a fixed $k$, the phase portrait of the transformations with the order $k$ determined by the reasonable stopping criterion; the relative error computed with the reasonable stopping criterion; and, the relative error at a fixed $k$. Phase portraits are an important qualitative tool in computational complex analysis because they plainly reveal roots, poles, and branch cuts. For a root, one observes a point in which the red-green-blue (RGB) colour wheel is traversed counterclockwise. For a pole, the RGB colour wheel is traversed clockwise. The contour plots of the relative error are clamped between $10^{-16}$ and $10^0$ to assist in the visualization of the plots. All plots are formed from numerical evaluations on $300\times300$ equispaced grids.

Comparing Figures~\ref{fig:1F1Drummond} and~\ref{fig:1F1Weniger}, we see that for the same order of transformation, $k=10$, the factorial Levin-type transformation includes a larger region with lower relative error than the Drummond transformation in approximating the entire function $\pFq{1}{1}$, which is in part explained by the poles of the former being further from the origin. The reasonable stopping criterion is less accurate in the right-half part of the complex plane.

With Figures~\ref{fig:2F1Drummond} and~\ref{fig:2F1Weniger}, we are reporting the approximations to the analytic function $\pFq{2}{1}$ in the cut plane $\C\setminus[1,\infty)$. In both cases, it is unfortunate that the spurious pole-root pairs do not reside on the function's branch cut. For the Drummond transformation, pole-root pairs appear to converge to the line $\Re z = \frac{1}{2}$, which has serious consequences for the transformation's ability to approximate the function to the right of this line: it appears to present a form of false converge to something completely fictitious but perhaps related to the function in hand. On the other hand, the pole-root pairs of the factorial Levin-type transformation appear to converge to the branch cut, enabling the reasonable stopping criterion to succeed in approximating the function with high relative accuracy all the way up to the cut. It then follows by regular perturbation theory that in general we cannot expect the spurious pole-root pairs of either transformation to live on the branch cut for any $\pFq{q+1}{q}$. Numerical experiments suggest that the description above is representative of the entire class, apart from select cases where the factorial Levin-type transformation's pole-root pairs do appear on the cut. We shall return to this point in the next section.

Regarding generalized hypergeometric functions analytic in the cut plane $\C\setminus[0,\infty)$, Figures~\ref{fig:2F0Drummond} and~\ref{fig:2F0Weniger} show how the transformations approximating $\pFq{2}{0}$ have their spurious root-pole pairs placed on the positive real axis. The phase portraits are sharp along the branch cut with the stopping criterion, though the accuracy of the Drummond transformation deteriorates most noticeably in the right-half complex plane.

Finally, Figures~\ref{fig:pFqDrummondOrderTime} and~\ref{fig:pFqWenigerOrderTime} illustrate, for the same three generalized hypergeometric functions, the order $k$ upon which the algorithm exits and the corresponding execution time in seconds, reported as an average over $10$ runs. The order $k$ is clamped between $10^1$ and $10^6$ and the time is clamped between $10^{-6}$ and $10^{-1}$ seconds to assist in the visualization of the plots. A close inspection shows that there is an approximately one-to-one correspondence between all plots in the top row and those directly below. Since the contours are plotted over the same relative range, we consider this to be a numerical proof that the linear complexity of the recurrence relations is realized in practice. For the Drummond transformation, we observe that the line $\Re z = \frac{1}{2}$ of accumulation of pole-root pairs of $\pFq{2}{1}$ is expensive, while the false convergence to the right of this line occurs at a much lower order. For the factorial Levin-type transformation, the order $k$ and execution time increase in the neighbourhood of the branch cuts of $\pFq{2}{1}$ and $\pFq{2}{0}$. It appears that both the order and execution times are nearly uniformly lower for the factorial Levin-type transformation than the Drummond transformation for all three examples.

\begin{figure}[htbp]
\begin{center}
\includegraphics[width=\textwidth]{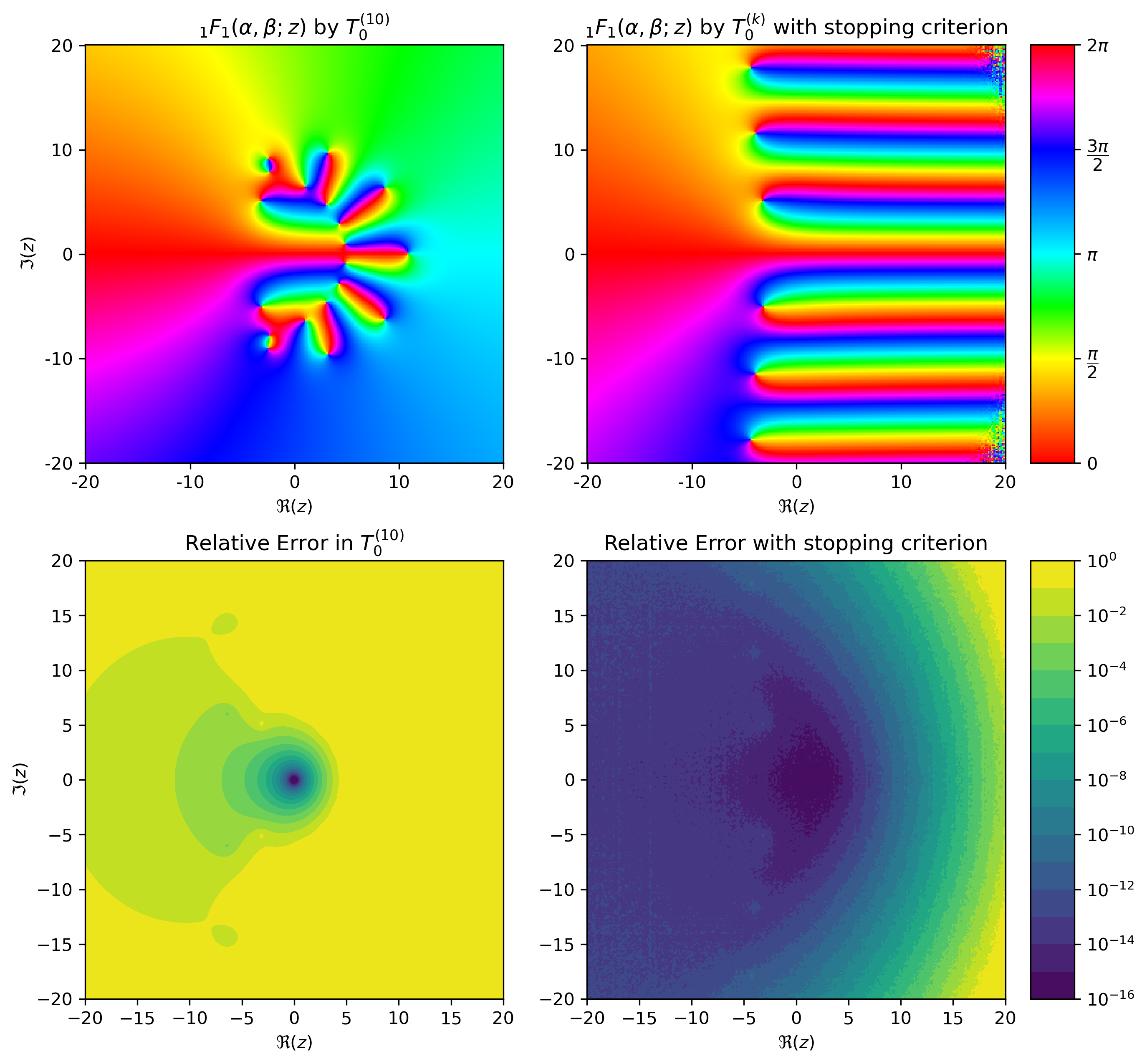}
\end{center}
\caption{The Drummond transformation for $\pFq{1}{1}(\alpha,\beta;z)$ with $\alpha=(5/4,)$ and $\beta=(3/2,)$.}
\label{fig:1F1Drummond}
\end{figure}

\begin{figure}[htbp]
\begin{center}
\includegraphics[width=\textwidth]{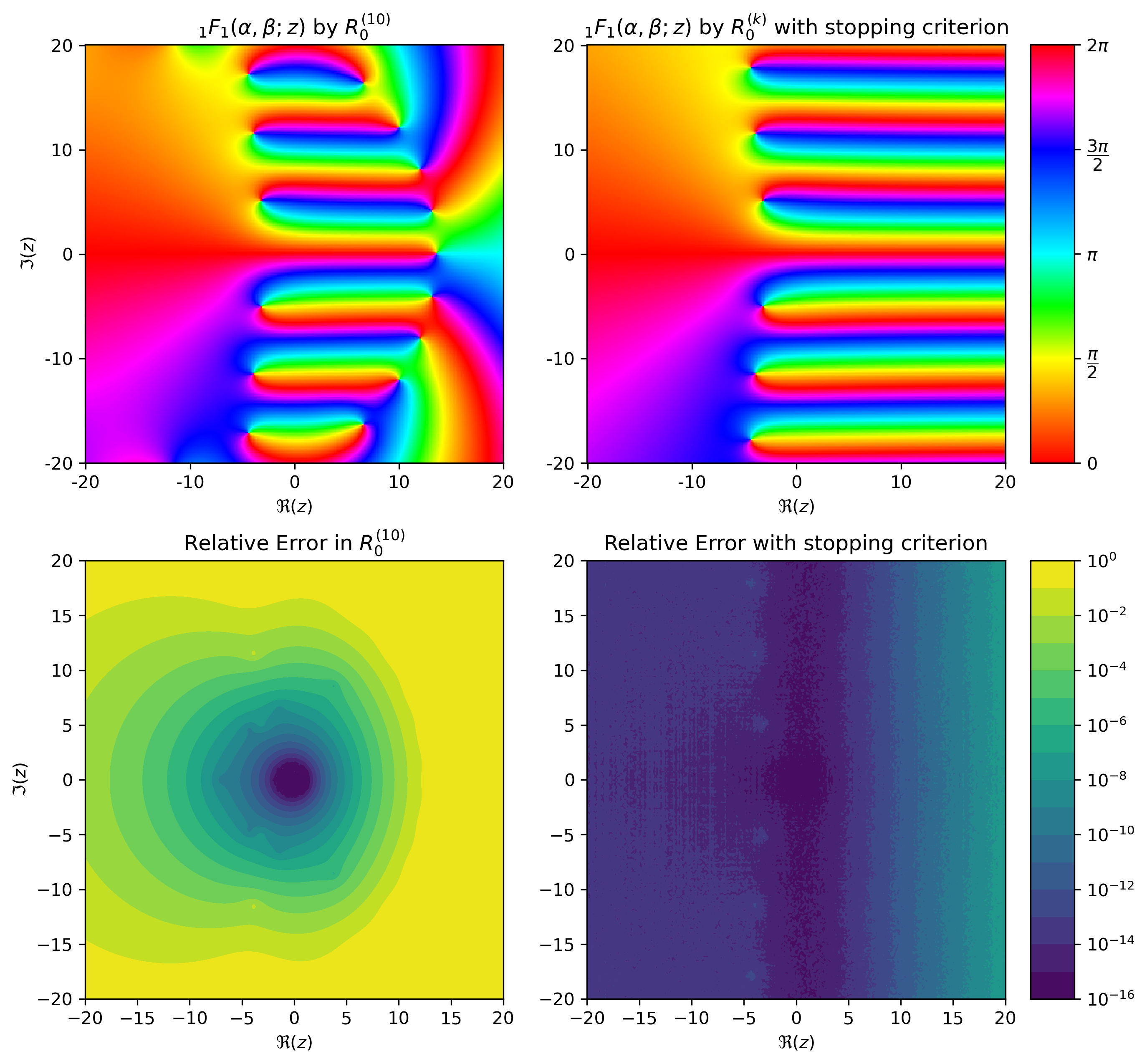}
\end{center}
\caption{The factorial Levin-type transformation for $\pFq{1}{1}(\alpha,\beta;z)$ with $\alpha=(5/4,)$ and $\beta=(3/2,)$.}
\label{fig:1F1Weniger}
\end{figure}

\begin{figure}[htbp]
\begin{center}
\includegraphics[width=\textwidth]{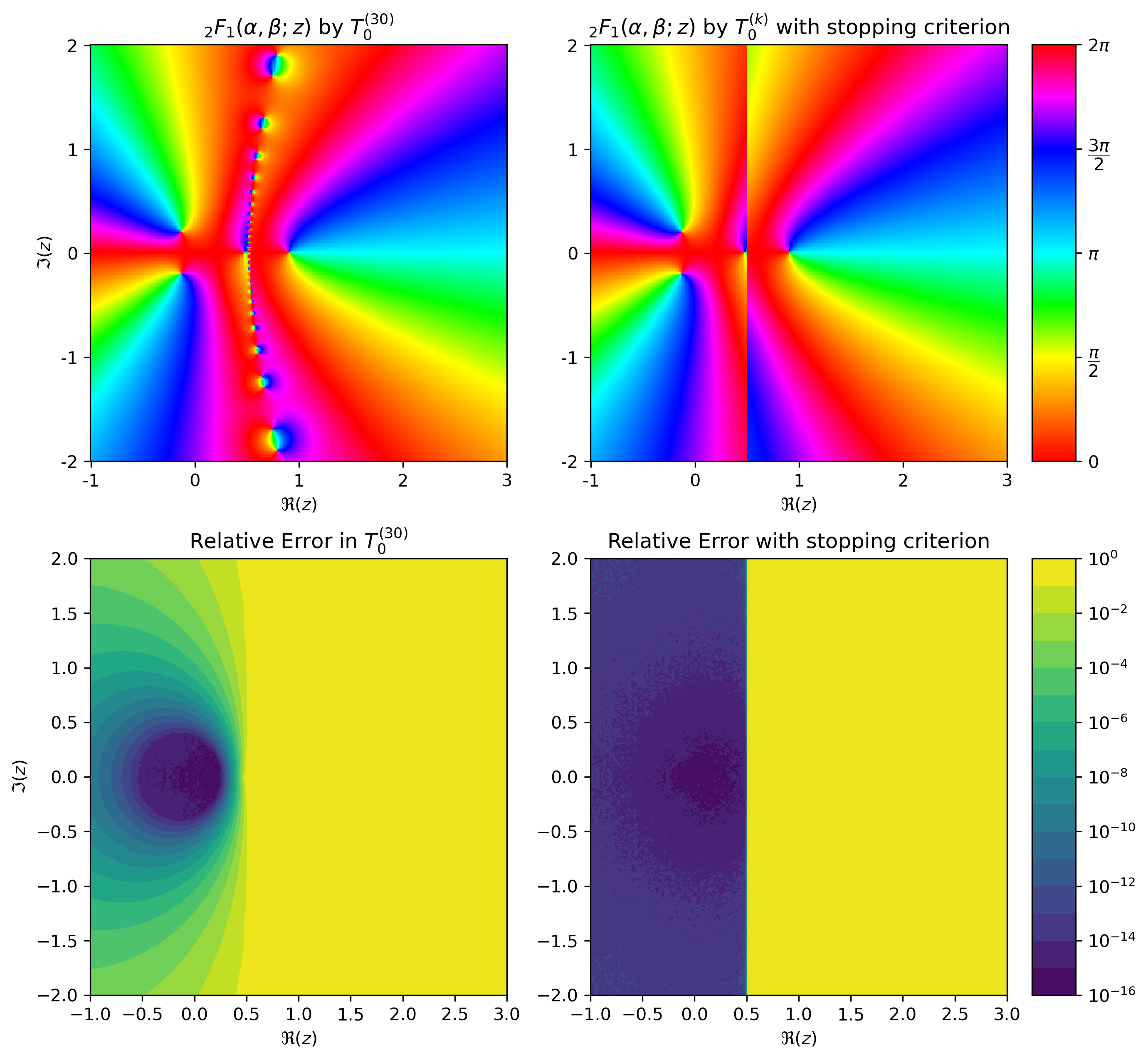}
\end{center}
\caption{The Drummond transformation for $\pFq{2}{1}(\alpha,\beta;z)$ with $\alpha=(1,-9/2)$ and $\beta=(-9/4,)$.}
\label{fig:2F1Drummond}
\end{figure}

\begin{figure}[htbp]
\begin{center}
\includegraphics[width=\textwidth]{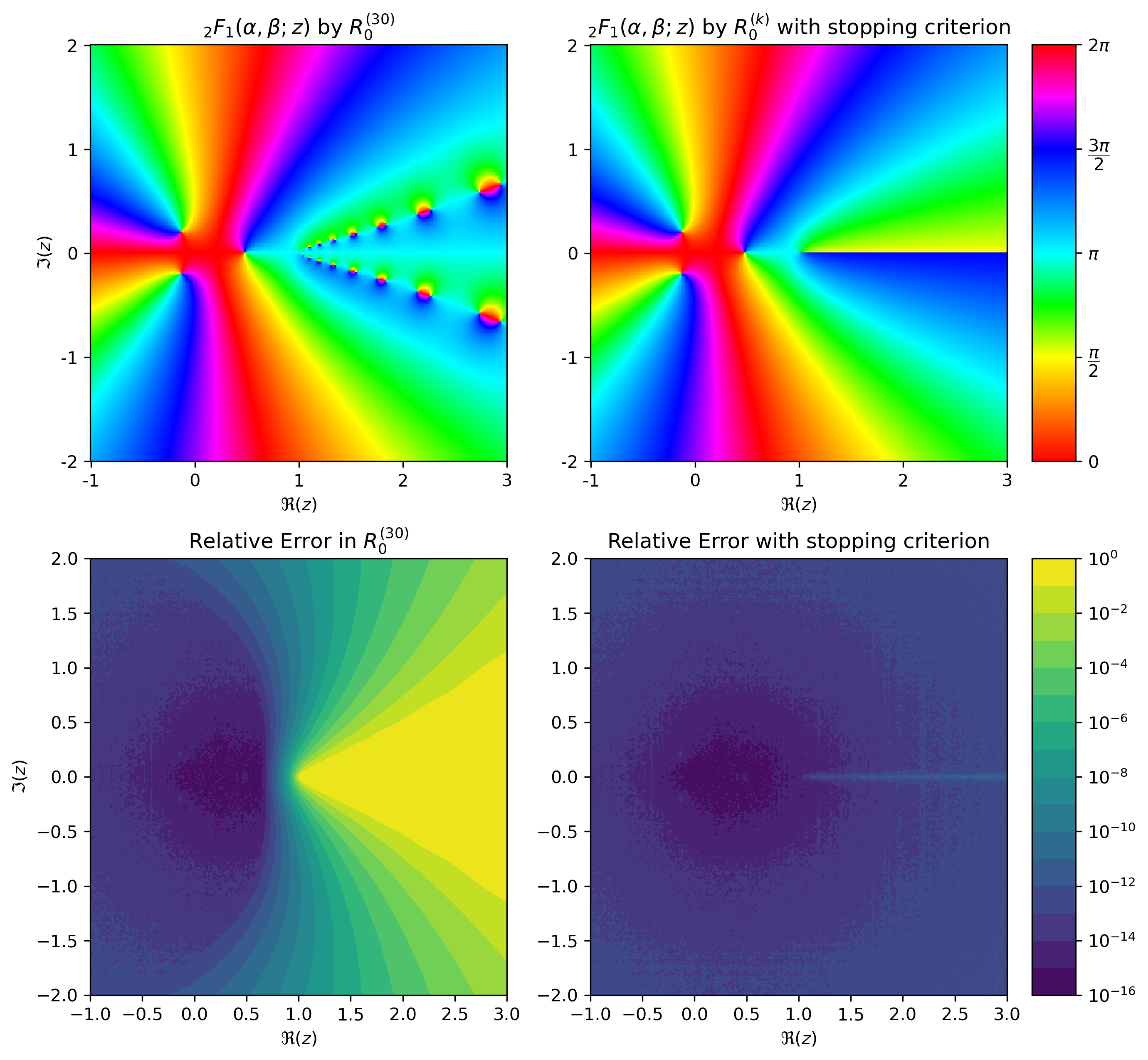}
\end{center}
\caption{The factorial Levin-type transformation for $\pFq{2}{1}(\alpha,\beta;z)$ with $\alpha=(1,-9/2)$ and $\beta=(-9/4,)$.}
\label{fig:2F1Weniger}
\end{figure}

\begin{figure}[htbp]
\begin{center}
\includegraphics[width=\textwidth]{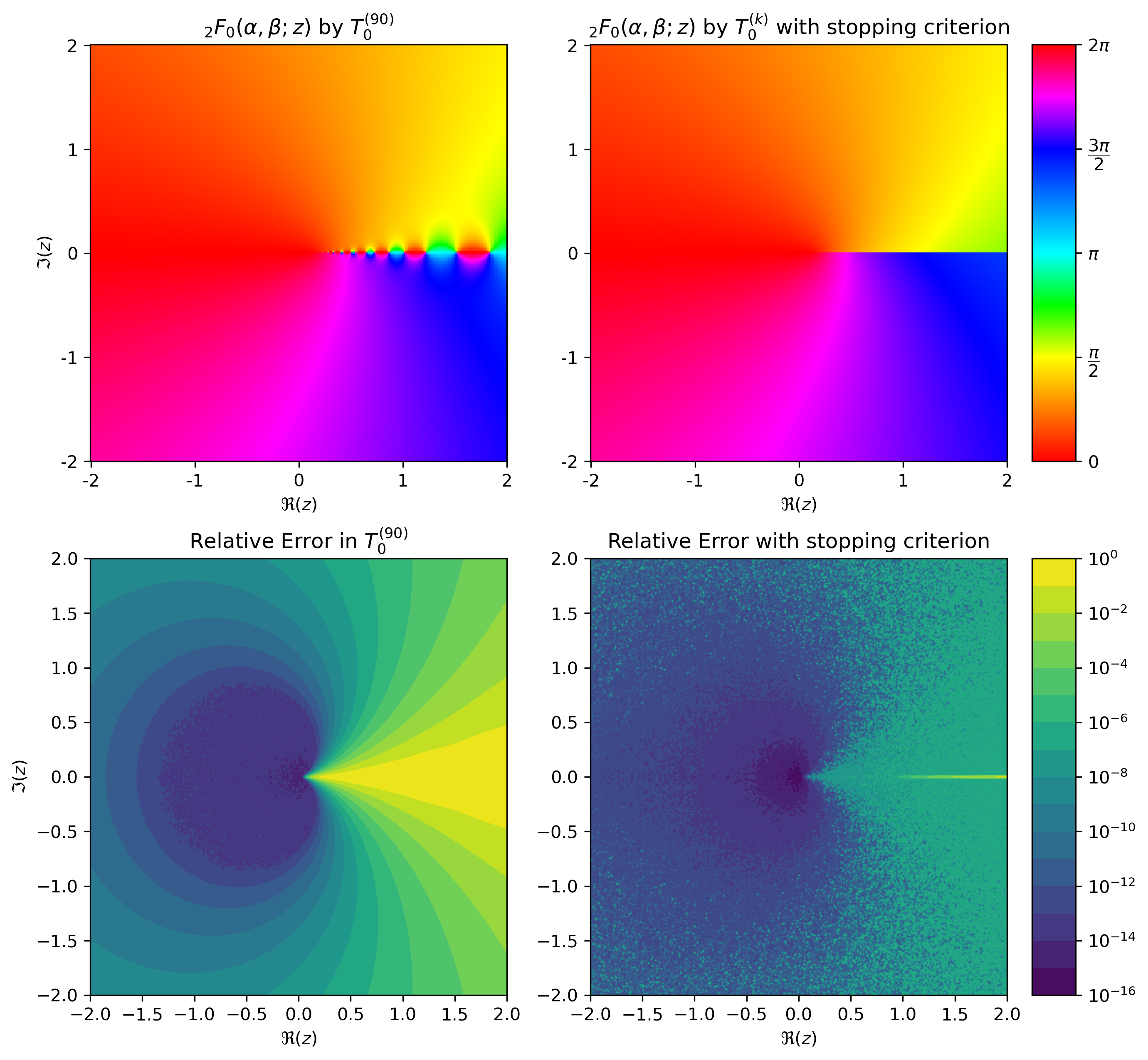}
\end{center}
\caption{The Drummond transformation for $\pFq{2}{0}(\alpha,\beta;z)$ with $\alpha=(1,3/2)$ and $\beta=()$.}
\label{fig:2F0Drummond}
\end{figure}

\begin{figure}[htbp]
\begin{center}
\includegraphics[width=\textwidth]{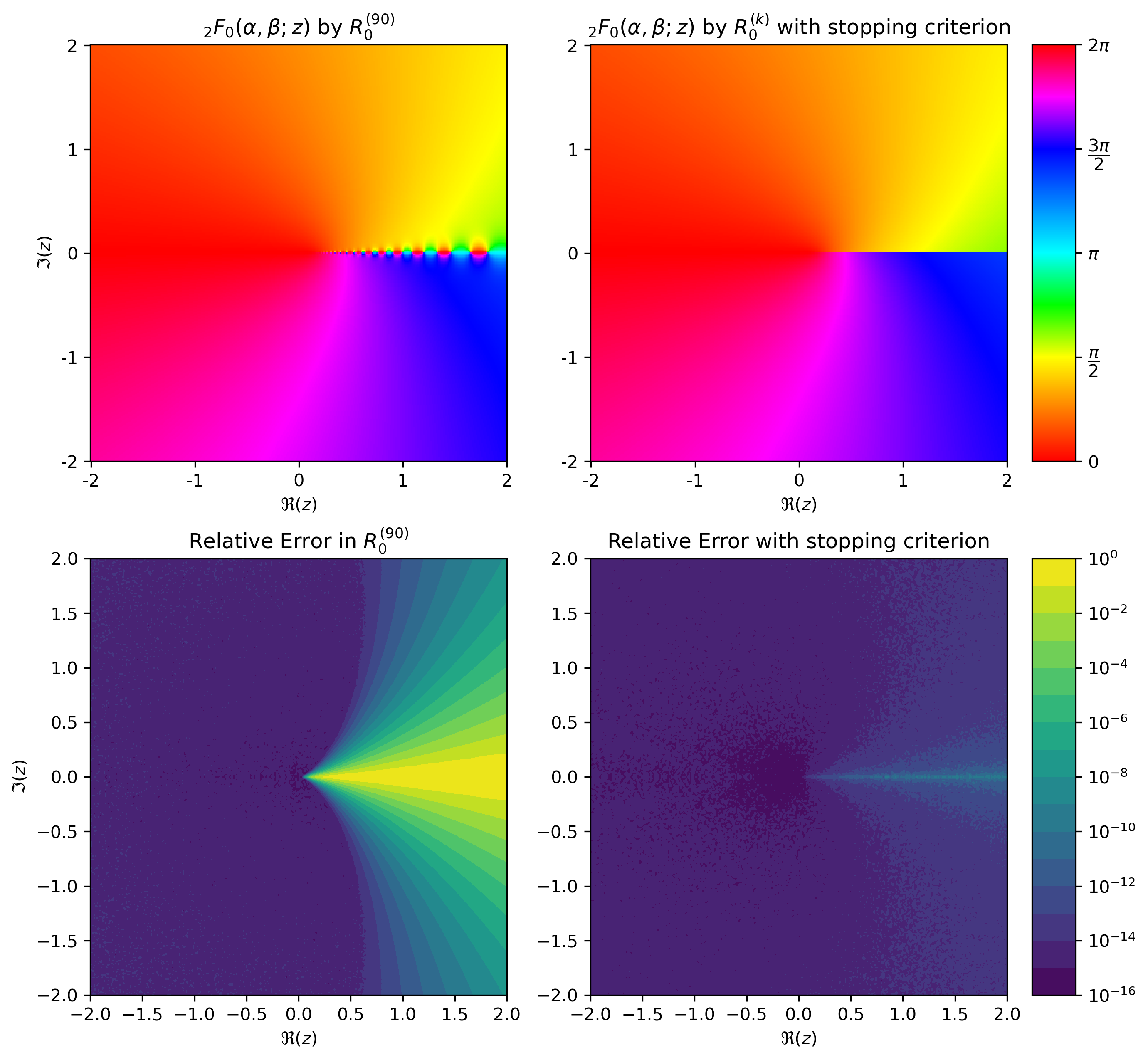}
\end{center}
\caption{The factorial Levin-type transformation for $\pFq{2}{0}(\alpha,\beta;z)$ with $\alpha=(1,3/2)$ and $\beta=()$.}
\label{fig:2F0Weniger}
\end{figure}

\begin{figure}[htbp]
\begin{center}
\includegraphics[width=\textwidth]{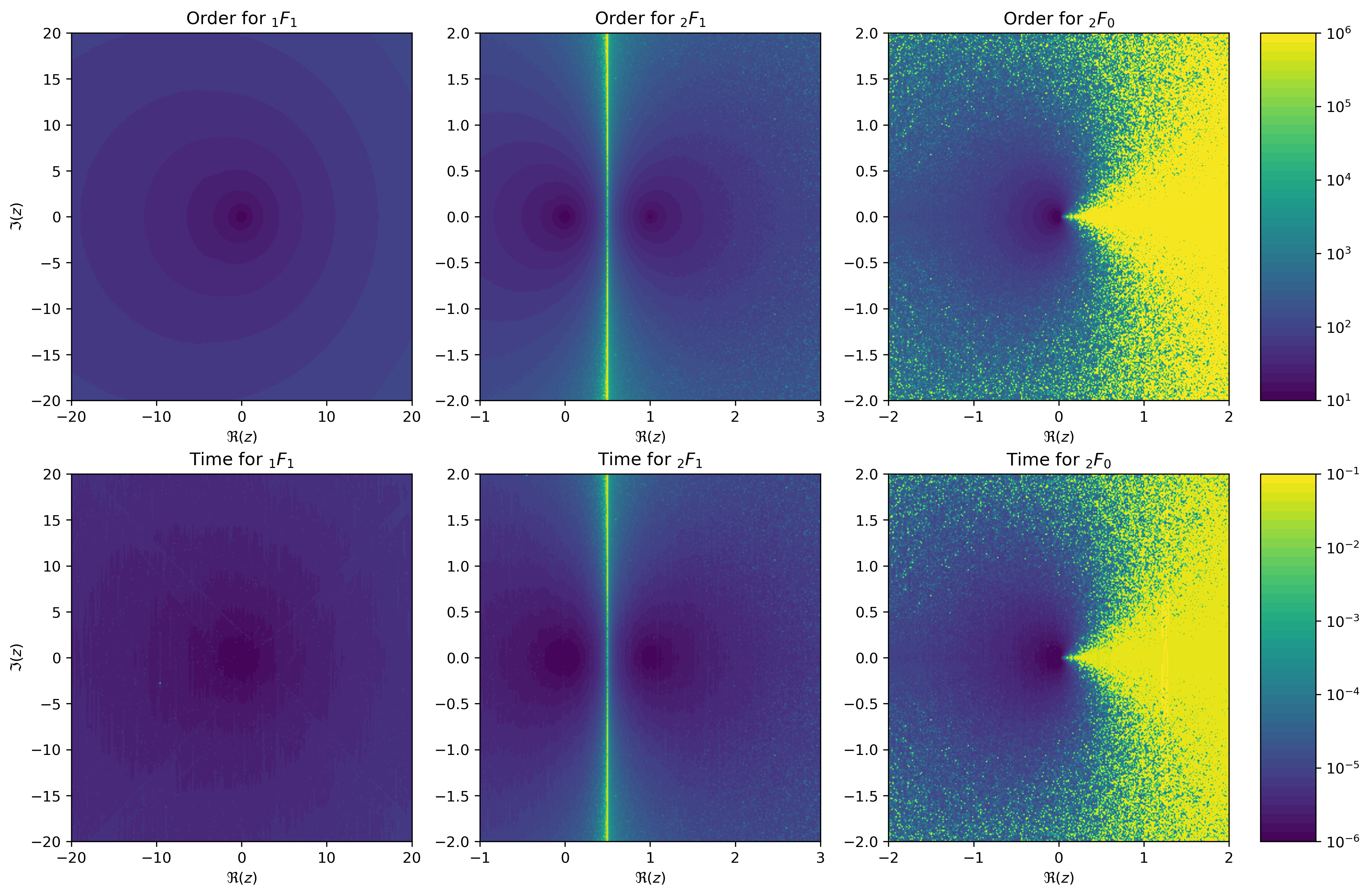}
\end{center}
\caption{The order $k$ (top row) and execution time in seconds (bottom row) of the Drummond transformation for the three generalized hypergeometric functions in Figures~\ref{fig:1F1Drummond}, \ref{fig:2F1Drummond}, and~\ref{fig:2F0Drummond}.}
\label{fig:pFqDrummondOrderTime}
\end{figure}

\begin{figure}[htbp]
\begin{center}
\includegraphics[width=\textwidth]{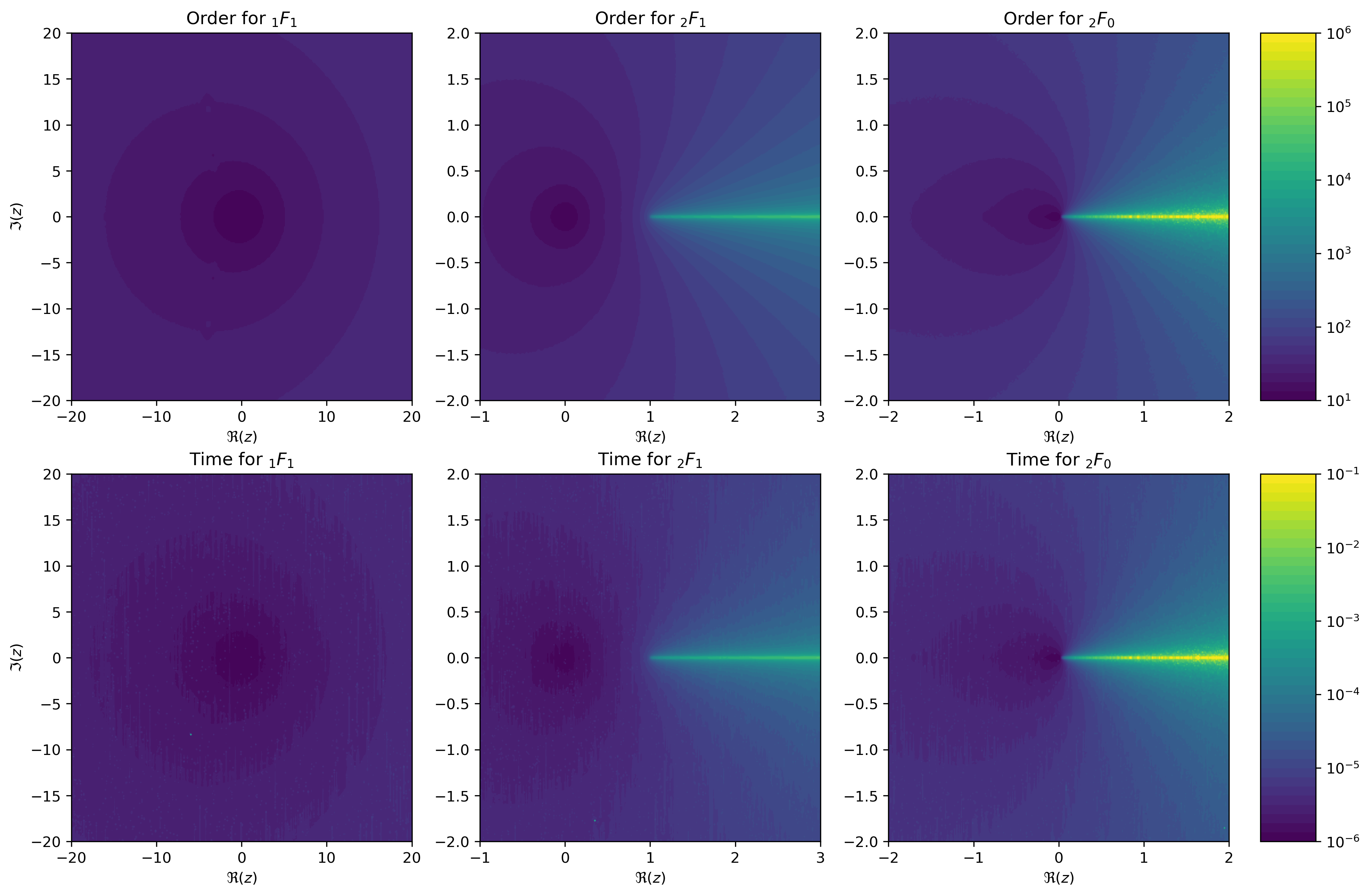}
\end{center}
\caption{The order $k$ (top row) and execution time in seconds (bottom row) of the factorial Levin-type transformation for the three generalized hypergeometric functions in Figures~\ref{fig:1F1Weniger}, \ref{fig:2F1Weniger}, and~\ref{fig:2F0Weniger}.}
\label{fig:pFqWenigerOrderTime}
\end{figure}

\section{On the placement of the poles}

It is important to be able to place the poles of the rational approximations of generalized hypergeometric functions. For a rational approximation to an entire function, we should expect all of the poles to scatter to infinity with increasing order $k$ in the type $(n+k,k)$. In the linear stability analysis of time-stepping schemes~\cite{Iserles-09}, it may also be important for the poles to live in one particular half of the complex plane. For a rational approximation to an analytic function with a branch cut, we should expect the poles of the rational approximation to be distributed along the branch cut or at the very least converge to it.

\subsection{P\'olya frequency functions}\label{subsection:PolyaFrequencyFunctions}

P\'olya frequency functions have been used by Borghi and Weniger~\cite{Borghi-Weniger-94-149-15} in the case of $\pFq{2}{0}(1,1;z)$ to locate the poles of both the Drummond and factorial Levin-type transformation (with $\gamma=1$) on the branch cut $(0,\infty)$. We begin our exploration with a generalization of their result.

\begin{lemma}[Richards~\cite{Richards-137/138-467-90} and Driver, Jordaan, and Mart\'inez-Finkelstein~\cite{Driver-Jordaan-Martinez-Finkelshtein-332-1045-07}]\label{eq:PolyaFrequencyFunctions}
Let $p\ge q+1$, $\alpha_1,\ldots,\alpha_p>0$ and $k_1,\ldots,k_{q+1}\in\N_0$. Then:
\[
\pFq{q+1}{p}\left(\begin{array}{c}\alpha_1+k_1,\ldots,\alpha_{q+1}+k_{q+1}\\\alpha_1,\ldots,\alpha_p\end{array}; z\right),
\]
is a P\'olya frequency function, and the associated terminating series:
\[
\pFq{q+2}{p}\left(\begin{array}{c}-k,\alpha_1+k_1,\ldots,\alpha_{q+1}+k_{q+1}\\\alpha_1,\ldots,\alpha_p\end{array}; z\right),
\]
has positive real roots.
\end{lemma}

For the particular choice of remainder estimate $\omega_n = \Delta s_n = a_{n+1}$, the denominators in both transformations of generalized hypergeometric functions are also related to terminating generalized hypergeometric functions.
\begin{lemma}\label{lemma:DenominatorsAsTerminatingpFq}
The denominators in the Drummond transformation of the generalized hypergeometric series are:
\begin{equation}
\Delta^k\left(\dfrac{1}{\Delta s_n}\right) = \dfrac{(-1)^k(\beta_1)_{n+1}\cdots(\beta_q)_{n+1}(1)_{n+1}}{z^{n+1}(\alpha_1)_{n+1}\cdots(\alpha_p)_{n+1}}\pFq{q+2}{p}\left(\begin{array}{c}-k, n+2, \beta_1+n+1,\ldots,\beta_q+n+1\\\alpha_1+n+1,\ldots,\alpha_p+n+1\end{array}; \frac{1}{z}\right).
\end{equation}
The denominators in the factorial Levin-type transformation of the generalized hypergeometric series are:
\begin{align}
\Delta^k\left(\dfrac{(n+\gamma)_{k-1}}{\Delta s_n}\right) & = \dfrac{(-1)^k(\beta_1)_{n+1}\cdots(\beta_q)_{n+1}(1)_{n+1}(n+\gamma)_{k-1}}{z^{n+1}(\alpha_1)_{n+1}\cdots(\alpha_p)_{n+1}}\nonumber\\
& \quad \pFq{q+3}{p+1}\left(\begin{array}{c}-k, k+n+\gamma-1, n+2, \beta_1+n+1,\ldots,\beta_q+n+1\\ n+\gamma,\alpha_1+n+1,\ldots,\alpha_p+n+1\end{array}; \frac{1}{z}\right),
\end{align}
and if $\gamma=2$, then:
\begin{equation}
\Delta^k\left(\dfrac{(n+2)_{k-1}}{\Delta s_n}\right) = \dfrac{(-1)^k(\beta_1)_{n+1}\cdots(\beta_q)_{n+1}(1)_{n+k}}{z^{n+1}(\alpha_1)_{n+1}\cdots(\alpha_p)_{n+1}}\pFq{q+2}{p}\left(\begin{array}{c}-k, k+n+1, \beta_1+n+1,\ldots,\beta_q+n+1\\ \alpha_1+n+1,\ldots,\alpha_p+n+1\end{array}; \frac{1}{z}\right).
\end{equation}
\end{lemma}
\begin{proof}
These formul\ae~follow from this representation of the binomial coefficient:
\[
\binom{k}{j}(-1)^j = \dfrac{(-k)_j}{j!},
\]
and the product of two Pochhammer symbols, $(\alpha)_{n+j+1} = (\alpha)_{n+1}(\alpha+n+1)_j$.
\end{proof}

The next two results are a direct consequence of Lemmata~\ref{eq:PolyaFrequencyFunctions} and~\ref{lemma:DenominatorsAsTerminatingpFq}.
\begin{theorem}\label{theorem:PFFDrummond}
Let $p > q+1$, let $\alpha_1+n+1,\ldots,\alpha_{p-1}+n+1>0$, let $\alpha_p\in\{-n,1-n,\ldots,1\}$, and let $\beta_i-\alpha_i\in\N_0$ for $i=1,\ldots,q$. Then the Drummond transformations with $\omega_n = \Delta s_n$ have positive real poles.
\end{theorem}
As examples, the Drummond transformations of $\pFq{2}{0}(\alpha,1;z)$ have positive real poles provided $\alpha+n+1>0$, generalizing the result of Borghi and Weniger~\cite[Lemma 6.1]{Borghi-Weniger-94-149-15}. And, the Drummond transformations of $\pFq{3}{0}(\alpha, \beta,1;z)$ have positive real poles provided $\alpha+n+1>0$ and $\beta+n+1>0$, proving with conditions the observation of Li and Slevinsky~\cite[\S 3.3]{Li-Slevinsky-397-108870-19}. Note that Theorem~\ref{theorem:PFFDrummond} introduces the restriction $p>q+1$, excluding the Gauss-like case $p=q+1$ in addition to the entire functions.

\begin{theorem}\label{theorem:PFFWeniger}
Let $p\ge q+1$, $\beta_i-\alpha_i\in\N_0$ for $i=1,\ldots,q$, and:
\begin{enumerate}
\item let $n+\gamma>0$, let $\alpha_1+n+1,\ldots,\alpha_{p-1}+n+1>0$, and let $\alpha_p\in\{-n,1-n,\ldots,1\}$; or,
\item if $\gamma=2$, let $\alpha_1+n+1,\ldots,\alpha_{p-1}+n+1>0$, and let $\alpha_p\in\{-n,1-n,\ldots,k\}$.
\end{enumerate}
Then the factorial Levin-type transformations with $\omega_n = \Delta s_n$ have positive real poles.
\end{theorem}
As an example, the factorial Levin-type transformations of $\pFq{2}{0}(\alpha,1;z)$ have positive real poles provided $\alpha+n+1>0$, generalizing the result of Borghi and Weniger~\cite[Theorem 5.1]{Borghi-Weniger-94-149-15}.

\subsection{Generalized eigenvalue problems}\label{subsection:GeneralizedEigenvalueProblems}

While P\'olya frequency functions are a powerful theoretical tool, this theory does not {\em characterize} the placement of the poles; in particular, due to the restriction $p\ge q+1$, the theory of P\'olya frequency functions does not apply to the rational approximants to entire hypergeometric functions. Moreover, if $p = q+1$, we expect the reciprocal poles to live more precisely in $(0,1]\subset(0,\infty)$. Therefore, we shall also consider how the recurrence relations from \S~\ref{section:linearcomplexity} translate to banded generalized eigenvalue problems for the reciprocal poles, $\zeta = z^{-1}$. Given that all three equations in Lemma~\ref{lemma:DenominatorsAsTerminatingpFq} include a factor of $\zeta^{n+1}$ in the denominators that cancels with an equal term in the numerators, all discussion in this section only concerns the nontrivial reciprocal poles.

Viewing reciprocal poles as generalized eigenvalues is useful for determining the partial fraction decomposition of a rational function. For example, with simple poles in hand, the residues may be computed by evaluating the numerator sequence at the poles in complex arithmetic and the denominator sequence at the poles with {\em dual}-complex arithmetic and extracting the derivatives through what is essentially forward-mode automatic differentiation.

Let $A\in\R^{k\times k}$ and $B\in\R^{k\times k}$ be two square matrices. Let $\lambda(A,B) = \{\zeta\in\C:\det(A-\zeta B) = 0\}$ denote the generalized eigenvalues of pencil $(A,B)$, and let $\lambda(A) = \lambda(A, I)$ denote the eigenvalues of $A$. Let $\displaystyle\rho(A) = \max_{\lambda\in\lambda(A)} \abs{\lambda}$ be the spectral radius of $A$.

From Theorem~\ref{theorem:DrummondRecurrence}, the denominators in the Drummond transformation with $\omega_n = \Delta s_n$ satisfy:
\begin{align}
\pFq{0}{0}(z):\quad D_n^{(k+1)} + D_n^{(k)} & = \zeta\left[(n+k+2)D_n^{(k)} + kD_n^{(k-1)}\right],\label{eq:Drummond0F0}\\
\pFq{1}{0}(\alpha;z):\quad (\alpha+n+k+1)D_n^{(k+1)} + (\alpha+n+2k+1)D_n^{(k)} + kD_n^{(k-1)} & = \zeta\left[(n+k+2)D_n^{(k)} + kD_n^{(k-1)}\right].\label{eq:Drummond1F0}
\end{align}

We shall also require a bound on the roots of generalized Bessel polynomials~\cite{Krall-Frink-65-100-49}:
\[
Y_k^{(\delta)}(z) = \sum_{j=0}^k\binom{k}{j}(k+\delta+1)_j)\left(\frac{-z}{2}\right)^j,
\]
given by Do\v cev~\cite{Docev-6-89-62} with a condition on $\delta$ removed by Saff and Varga~\cite[Theorems 5.1 \& 5.2]{Saff-Varga-7-344-76}.
\begin{lemma}
If $k+\delta+1>0$, then all the zeros of $Y_k^{(\delta)}(z)$ lie in the closed disk:
\[
\left\{z\in\C:\abs{z}\le \frac{2}{k+\delta+1}\right\}.
\]
\end{lemma}

\begin{theorem}\label{theorem:Drummond0F0}
For the $k$ reciprocal roots of $D_n^{(k)}$ in Eq.~\eqref{eq:Drummond0F0}, we have:
\[
\frac{1}{n+k+1} \le \max_{1\le j\le k} \abs{\zeta_j} \le \frac{1}{n+2}.
\]
\end{theorem}
\begin{proof}
By Eq.~\eqref{eq:Drummond0F0}, the reciprocal roots of $D_n^{(k)}$ correspond to the generalized eigenvalues of the pencil $(A,B)$:
\[
\underbrace{\begin{pmatrix} 1 & 1\\ & \ddots & \ddots\\ & & \ddots & 1\\ & & & 1\end{pmatrix}}_{=:A}v = \zeta\underbrace{\begin{pmatrix} n+2\\ 1 & \ddots\\ & \ddots & \ddots\\ & & k-1 & n+k+1 \end{pmatrix}}_{=:B}v,
\]
and, since $B$ is invertible, the associated regular eigenvalue problem $B^{-1}A$. Since:
\[
\frac{\abs{\tr(B^{-1}A)}}{k} \le \rho(B^{-1}A),
\]
we easily obtain the lower bound with $B^{-1}$:
\[
(B^{-1})_{i,j} = \left\{ \begin{array}{ccc} \dfrac{(-1)^{i-j} (j)_{i-j}}{(n+j+1)_{i-j+1}} & {\rm for} & i \ge j,\\ 0 & \multicolumn{2}{c}{\rm otherwise.}\end{array}\right.
\]
Since:
\begin{align*}
\tr(B^{-1}A) & = \sum_{i,j=1}^k (B^{-\top})_{i,j}A_{i,j},\\
& = \sum_{j=1}^{k-1}\left[ (B^{-1})_{j,j} + (B^{-1})_{j+1,j} \right] + (B^{-1})_{k,k},\\
& = \sum_{j=1}^{k-1} \left[\frac{1}{n+j+1} - \frac{j}{(n+j+1)_2}\right] + \frac{1}{n+k+1},\\
& = \sum_{j=1}^{k-1} \frac{n+2}{(n+j+1)_2} + \frac{1}{n+k+1},\\
& = (n+2)\left(\frac{1}{n+2}+\frac{1}{n+k+1}\right) + \frac{1}{n+k+1} = \frac{k}{n+k+1}.
\end{align*}
The upper bound follows from the representation of the denominators as generalized Bessel polynomials:
\[
D_n^{(k)} = (-1)^k(1)_{n+1}\zeta^{n+1} Y_k^{(n+1-k)}(2\zeta).
\]
\end{proof}

\begin{theorem}\label{theorem:Drummond1F0}
If $\alpha\le1$ and $\alpha+n+1>0$, the $k$ reciprocal roots of $D_n^{(k)}$ in Eq.~\eqref{eq:Drummond1F0} are all contained in the closed disk:
\[
\left\{\zeta\in\C:\abs{\zeta-1}\le 1\right\}.
\]
\end{theorem}
\begin{proof}
By Eq.~\eqref{eq:Drummond1F0}, the reciprocal roots of $D_n^{(k)}$ are the generalized eigenvalues of the pencil $(A,B)$:
\[
\underbrace{\begin{pmatrix} \alpha+n+1 & \alpha+n+1\\ 1 & \alpha+n+3 & \alpha+n+2\\ & \ddots & \ddots & \alpha+n+k-1\\ & & k-1 & \alpha+n+2k-1\end{pmatrix}}_{=:A}v = \zeta\underbrace{\begin{pmatrix} n+2\\ 1 & \ddots\\ & \ddots & \ddots\\ & & k-1 & n+k+1 \end{pmatrix}}_{=:B}v,
\]
Shifting, $\eta = \zeta-1$:
\[
\begin{pmatrix} \alpha-1 & \alpha+n+1\\ & \ddots & \ddots\\ & & \ddots & \alpha+n+k-1\\ & & & \alpha+k-2\end{pmatrix}v = \eta\begin{pmatrix} n+2\\ 1 & \ddots\\ & \ddots & \ddots\\ & & k-1 & n+k+1 \end{pmatrix}v,
\]
Inverting $B$ as above and applying it to the matrix on the left, we find a matrix with almost all absolute row sums equal to $1$, and the last is less than or equal to $1$ only if $\alpha\le 1$. Specifically, let $D = B^{-1}(A-B)$ so that:
\[
D_{i,j} = \left\{ \begin{array}{ccc} \dfrac{(-1)^{i-j} (j)_{i-j}(\alpha-1)(n+1)}{(n+j)_{i-j+2}} & {\rm for} & j\le i,\\ \frac{n+i+\alpha}{n+i+1} & {\rm for} & j = i+1\\ 0 & \multicolumn{2}{c}{\rm otherwise.}\end{array}\right.
\]
Then, since $\rho(D)\le \norm{D}_\infty$ and:
\[
\sum_{j=1}^i \frac{(j)_{i-j}(n+1)}{(n+j)_{i-j+2}} = \frac{1}{n+i+1},
\]
we have:
\[
\norm{D}_\infty = \max_{1\le i\le k}\left\{ \frac{|n+i+\alpha|}{n+i+1}(1-\delta_{k,i}) + \frac{|\alpha-1|}{n+i+1}\right\}
\]
Since $\alpha<1$ and $n+\alpha+1\ge0$:
\[
\norm{D}_\infty = \max_{1\le i\le k}\left\{ \frac{(n+i+\alpha)(1-\delta_{k,i}) + 1-\alpha}{n+i+1}\right\} = \max\left\{ 1, \frac{1-\alpha}{n+k+1}\right\} = 1.
\]
\end{proof}
From Theorem~\ref{theorem:FactorialLevinRecurrence} and Remark~\ref{remark:reducelength}, the denominators in the factorial Levin-type transformation with $\omega_n = \Delta s_n$ and $\gamma = 2$ satisfy:
\begin{align}
\pFq{0}{0}(z):\quad \frac{Q_n^{(k+1)}}{(n+2k+1)_2} + \frac{n Q_n^{(k)}}{(n+2k)(n+2k+2)} - \frac{k(n+k)Q_n^{(k-1)}}{(n+2k)_2} & = \zeta Q_n^{(k)},\label{eq:Delta0F0}\\
\pFq{1}{0}(\alpha;z):\quad \frac{(n+k+\alpha+1)Q_n^{(k+1)}}{(n+2k+1)_2} +\frac{1}{2}\left[1+\frac{n(n+2\alpha)}{(n+2k)(n+2k+2)}\right]Q_n^{(k)} + \frac{(k-\alpha)k(n+k)}{(n+2k)_2}Q_n^{(k-1)} & = \zeta Q_n^{(k)}.\label{eq:Delta1F0}
\end{align}

\begin{theorem}\label{theorem:Delta0F0}
For the $k$ reciprocal roots of $Q_n^{(k)}$ in Eq.~\eqref{eq:Delta0F0}, we have:
\[
\frac{1}{n+2k} \le \max_{1\le j\le k} \abs{\zeta_j} \le \frac{1}{n+k+1}.
\]
\end{theorem}
\begin{proof}
Eq.~\eqref{eq:Delta0F0} defines a tridiagonal eigenvalue problem, and we take its trace for the lower bound:
\[
\frac{\abs{\tr(A)}}{k} = \frac{1}{k}\sum_{j=0}^{k-1}\frac{n}{(n+2j)(n+2j+2)} = \frac{1}{n+2k} \le \rho(A).
\]
The upper bound follows from the representation of the denominators as generalized Bessel polynomials:
\[
Q_n^{(k)} = (-1)^k(1)_{n+k}\zeta^{n+1} Y_k^{(n)}(2\zeta).
\]
\end{proof}

\begin{theorem}\label{theorem:Delta1F0}
For $-n-1<\alpha<1$, the $k$ reciprocal roots of $Q_n^{(k)}$ in Eq.~\eqref{eq:Delta1F0} are all real and contained in $(0,1)$.
\end{theorem}
\begin{proof}
Eq.~\eqref{eq:Delta1F0} defines a real irreducible tridiagonal eigenvalue problem on which Gerschgorin's theorem may be applied after balancing $1$- and $\infty$-norms. The recurrence relation is solved by scaled Jacobi polynomials~\cite[\S 18.3]{Olver-et-al-NIST-10}:
\[
Q_n^{(k)} = \dfrac{(-1)^kk!(n+k)!}{(\alpha)_{n+k+1}}\zeta^{n+1}P_k^{(\alpha+n,-\alpha)}(1-2\zeta),
\]
whose roots lie in $(-1,1)$.
\end{proof}

Comparing Theorems~\ref{theorem:Drummond0F0} and~\ref{theorem:Delta0F0}, we observe that the rate at which the reciprocal roots of $Q_n^{(k)}$ tend to $0$ is characterized as $\OO(k^{-1})$, while the upper bound for the reciprocal roots of $D_n^{(k)}$ is not as informative. In this case, we have evidence that the factorial Levin transformation is superior to the Drummond transformation given that the lower bound on the largest reciprocal root of $D_n^{(k)}$ is the same as the upper bound of those of $Q_n^{(k)}$.

Unfortunately, the roots of $D_n^{(k)}$ in Eq.~\eqref{eq:Drummond1F0} appear to be misplaced: instead of lying on the ray $[1,\infty)$, they appear to be complex and as $k\to\infty$ tend to the boundary of the disk $\abs{\zeta-1} = 1$ from the right, which is equivalent to the line $\Re z = \tfrac{1}{2}$. Figure~\ref{fig:2F1Drummond} shows numerical evidence to suggest that this phenomenon carries over to $\pFq{2}{1}$ as well. While Theorem~\ref{theorem:Delta1F0} is an unequivocal success, Figure~\ref{fig:2F1Weniger} shows that this results does not carry over to $\pFq{2}{1}$.

\section{Conclusion and Future Directions}

Recurrence relations are derived for certain rational approximations to generalized hypergeometric functions coming from the Drummond and factorial Levin-type sequence transformations. Numerical results illustrate that the latter are superior in each of the three numerical representative classes. In the context of a package for numerically computing generalized hypergeometric functions, these algorithms can only be seen as a special but important subset of the whole. A polyalgorithm could include special cases for the parameters, use of the Maclaurin series, linear, quadratic, and cubic variable transformations, contiguous relations~\cite{Wimp-22-363-68}, continued fractions, quadrature, and solution of the differential equation~\cite[\S 16.8.3]{Olver-et-al-NIST-10}, among many other reasonable approaches. Assembling such a Frankenstein is beyond the scope of this work.

We will now identify a few avenues for future work. The instability in the original formul\ae~for the Drummond and factorial Levin-type transformations is also present when summing sequences where the ratio of successive terms is irrational, such as the Fox--Wright function~\cite{Fox-27-389-28,Wright-10-286-35}, defined by its Maclaurin series:
\[
{}_p\Psi_q\left[\begin{array}{c} (a_1, A_1),\ldots,(a_p, A_p)\\ (b_1,B_1),\ldots,(b_q,B_q)\end{array}; z\right] = \sum_{k=0}^\infty\dfrac{\Gamma(a_1+A_1k)\cdots\Gamma(a_p+A_pk)}{\Gamma(b_1+B_1k)\cdots\Gamma(b_q+B_qk)}\frac{z^k}{k!},
\]
or the Riemann zeta function~\cite[\S 25.2]{Olver-et-al-NIST-10}, defined for $\Re s > 1$ by its infinite series:
\[
\zeta(s) = \sum_{n=1}^\infty \frac{1}{n^s}.
\]
It can be shown that the terms of these sequences still satisfy Theorem~\ref{theorem:mthorderdifference} with $m=1$. While Theorems~\ref{theorem:DrummondRecurrence} and~\ref{theorem:FactorialLevinRecurrence} do not apply unless $A_i,B_i\in\Z$ in the former or $s\in\Z$ in the latter, it is worth exploring whether Eqs.~\eqref{eq:DrummondDenominatorRecurrence},~\eqref{eq:DrummondNumeratorRecurrence},~\eqref{eq:WenigerDenominatorFiniteRecurrence}, and~\eqref{eq:WenigerNumeratorFiniteRecurrence} can be adapted to this setting and provide more stable implementations.

The Drummond and Levin-type transformations have continuous analogues~\cite{Levin-Sidi-9-175-81,Gray-Wang-29-271-92,Brezinski-Redivo-Zaglia-60-1231-10,Slevinsky-Safouhi-48-301-08,Slevinsky-Safouhi-109-1741-09,Slevinsky-Safouhi-60-1411-10,Gaudreau-Slevinsky-Safouhi-34-B65-12}, the latter of which is known as the $G_n^{(1)}$ transformation for approximating semi-infinite integrals. A particular case~\cite{Slevinsky-Safouhi-92-973-23} of linear complexity recurrence relations has been constructed for incomplete Bessel functions. The next logical step is a systematic construction and investigation based on integrands satisfying first-order linear homogeneous differential equations with a rational variable coefficient.

Finally, hypergeometric functions of matrix argument have innumerable applications. Two particular ones include the celebrated $\varphi$-functions in exponential integrators~\cite{Hochbruck-Ostermann-19-209-10} and computing $A^\alpha$, $\log(A)$, among others~\cite{Hale-Higham-Trefethen-46-2505-08}. For $\exp(A)$, unitarity preserving approximations are non-negotiable in the time evolution of physical observables in quantum mechanics. It is well known~\cite{Birkhoff-Varga-44-1-65,Ehle-4-671-73,Saff-Varga-25-1-75} that type-$(k,k)$ Pad\'e approximants to the exponential function are unitary, since each of its poles in the right-half plane is paired with a root in the left-half plane, preserving the property that $|[k/k]_{\exp}(z)| = 1$ for $z\in\ii\R$. In~\cite[Example 1]{Sidi-7-37-81}, Sidi shows that the factorial Levin-type transformation with $\omega_n = a_{n+1}$ and $\gamma=2$ of $\exp(z)$ creates Pad\'e approximants. By Theorem~\ref{theorem:FactorialLevinRecurrence} and Eq.~\eqref{eq:Delta0F0}, our recurrence for the main diagonal in the Pad\'e table to the exponential is extremely efficient on the imaginary axis:
\begin{equation}\label{eq:Padeforexp}
\begin{array}{lll}
\mu_0 = 1, & \mu_1 = \dfrac{1}{2-z}, & \mu_{k+1} = \dfrac{1}{4k+2 + z^2\mu_k},\\&&\\
R_0 = 1, & R_1 = R_0 + 2z\mu_1, & R_{k+1} = [(4k+2)R_k + z^2R_{k-1}\mu_k]\mu_{k+1}.
\end{array}
\end{equation}
Figure~\ref{fig:Padeforexp} depicts the type-$(k,k)$ Pad\'e approximants to $\exp(z)$ on the imaginary axis. It is observed that convergence takes places far from the origin and the scheme is stable in that divergence is not observed if the order is greater than necessary.
\begin{figure}[htbp]
\begin{center}
\includegraphics[width=\textwidth]{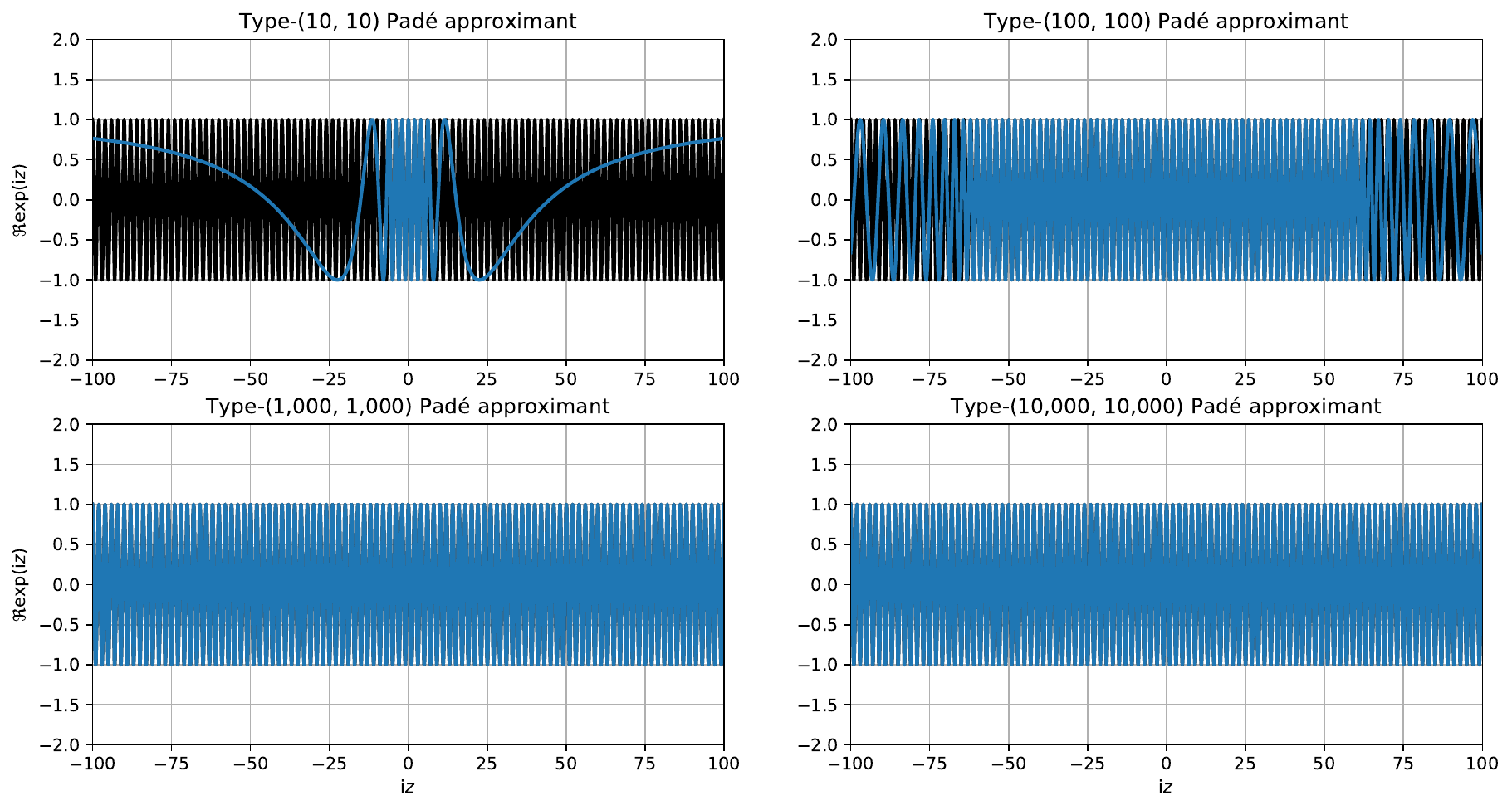}
\end{center}
\caption{The real part of unitarity-preserving type-$(k,k)$ Pad\'e approximants to $\exp(z)$ computed by Eq.~\eqref{eq:Padeforexp}.}
\label{fig:Padeforexp}
\end{figure}
For a more complete appreciation, our algorithm approximates $\exp(10^9\ii)$ in double precision by:
\[
[500,004,886/500,004,886]_{\exp}(10^9\ii) = 0.837\,887\,181\,360\,9+0.545\,843\,449\,454\,380\,2\ii,
\]
a result with an absolute and relative error less than $6.43\times10^{-12}$, and an error in unitarity less than $5.86\times10^{-13}$.

According to Moler and Van Loan~\cite{Moler-Van-Loan-20-801-78}, while Pad\'e approximation combined with scaling and squaring is the most effective algorithm to compute the matrix exponential, most of the computational time is spent in the latter step since classical computation of Pad\'e numerators and denominators suffers from significant rounding errors. With Eq.~\eqref{eq:Padeforexp}, this balance may be reconsidered.

More stable algorithms for diagonals in the Pad\'e table of other functions could conceivably be constructed by recurrence relations analogous to those of the Drummond and factorial Levin-type transformations in this work.

\section*{Acknowledgments}

We thank Nick Trefethen for a discussion on Pad\'e approximation to $e^z$. RMS is supported by the Natural Sciences and Engineering Research Council of Canada, through a Discovery Grant (RGPIN-2017-05514).

\bibliography{/Users/Mikael/Bibliography/Mik}

\end{document}